\documentclass[11pt]{amsart}

\usepackage{graphicx}
\usepackage[margin=1in]{geometry}
\usepackage{enumerate}
\usepackage{xcolor}

\newtheorem{theorem}{Theorem}[section]
\newtheorem{proposition}[theorem]{Proposition}
\newtheorem{lemma}[theorem]{Lemma}
\newtheorem{corollary}{Corollary}[theorem]

\theoremstyle{definition}
\newtheorem{definition}[theorem]{Definition}
\newtheorem{example}[theorem]{Example}

\theoremstyle{remark}
\newtheorem{remark}[theorem]{Remark}

\newcommand{\SP}{\text{SP}\,}
\newcommand{\SG}{\text{SG}\,}
\DeclareMathOperator{\CDS}{\textbf{CDS}}

\numberwithin{equation}{section}

\newtheorem*{repp@theorem}{\repp@title (reformulated)}
\newcommand{\newrepptheorem}[2]{%
\newenvironment{repp#1}[1]{%
 \def\repp@title{#2 \ref{##1}}%
 \begin{repp@theorem}}%
 {\end{repp@theorem}}}
\makeatother
\newrepptheorem{theorem}{Theorem}

\theoremstyle{definition}
\makeatletter
\newtheorem*{rep@theorem}{\rep@title \ continued}
\newcommand{\newreptheorem}[2]{%
\newenvironment{rep#1}[1]{%
 \def\rep@title{#2 \ref{##1}}%
 \begin{rep@theorem}}%
 {\end{rep@theorem}}}
\makeatother
\newreptheorem{example}{Example}

\begin{document}

\title{Classifying Permutations under Context-Directed Swaps and the \textbf{cds} game}

\author{G. Brown}
\address{(Garrett Brown) Harvard University}
\curraddr{153 Kirkland Mail Center, 95 Dunster Street, Cambridge, MA, 02138, USA}
\email{garrettbrown@college.harvard.edu}

\author{A. Mitchell}
\address{Ohio State University}
\curraddr{1922 Summit St, Columbus, OH, 43201}
\email{mitchell.1583@osu.edu}

\author{R. Raghavan}
\address{Ohio State University}
\curraddr{Columbus, OH, 43210}
\email{raghavan.63@osu.edu}

\author{J. Rogge}
\address{University of Illinois at Urbana-Champaign}
\curraddr{2557 W Fargo Ave, Chicago IL, 60645}
\email{jwrogge2@illinois.edu}

\author{M. Scheepers}
\address{Boise State University}
\curraddr{1910 University Drive, Boise, ID, 83725}
\email{mscheepe@boisestate.edu}


\date{October 31, 2020} 

\begin{abstract}
A special sorting operation called Context Directed Swap, and denoted \textbf{cds}, performs certain types of block interchanges on permutations. 
When a permutation is sortable by \textbf{cds}, then \textbf{cds} sorts it using the fewest possible block interchanges of any kind.  
This work introduces a classification of permutations based on their number of \textbf{cds}-eligible contexts.
In prior work an object called the strategic pile of a permutation was discovered and shown to provide an efficient measure of the non-\textbf{cds}-sortability of a permutation. 
Focusing on the classification of permutations with maximal strategic pile,  a complete characterization is given when the number of \textbf{cds}-eligible contexts is close to maximal as well as when the number of eligible contexts is minimal. A group action that preserves the number of \textbf{cds}-eligible contexts of a permutation provides, via the orbit-stabilizer theorem, enumerative results regarding the number of permutations with maximal strategic pile and a given number of \textbf{cds}-eligible contexts.  
Prior work introduced a natural two-person game on permutations that are not \textbf{cds}-sortable. The decision problem of which player has a winning strategy in a particular instance of the game appears to be of high computational complexity. Extending prior results, this work presents new conditions for player ONE to have a winning strategy in this combinatorial game. 
\end{abstract}

\maketitle

\section{Introduction}
The sortability of \textit{permutations}, non-repetitve arrays of integers, is of interest to a variety of fields including scientific computing and genetics.  
We study a a particular block-interchange sorting operation postulated to occur in the genomic sorting of single-celled organisms called \textit{ciliates} \cite{PRESCOTT}.  This operation participates in decrypting the ciliate's micronuclear genome to construct a new macronucleus.  We refer to this block-interchange operation as \textbf{cds}, abbreviating ``context directed swap".

A permutation $\pi$ is \textit{\textbf{cds}-sortable} if successive applications of \textbf{cds} on $\pi$ results in the identity permutation.  Not all permutations are \textbf{cds}-sortable.  The criteria for sortability were studied previously in \cite{ADAMYK, PETRE1, PETRE2} and others.  In \cite{Christie} Christie discovered that \textbf{cds} is a \textit{minimal block-interchange}, sorting a \textbf{cds}-sortable permutation using the fewest possible block interchanges.  If a permutation is not \textbf{cds}-sortable, successive applications of \textbf{cds} will result in a permutation where \textbf{cds} no longer applies.  Define a permutation where \textbf{cds} does not apply as a \textbf{cds} \textit{fixed point}.  The set of \textbf{cds} fixed points reachable from a permutation $\pi$, excluding the identity, is called the \textit{strategic pile} of the permutation, and is denoted $SP(\pi)$.  An interesting phenomenon arises when a permutation is not \textbf{cds}-sortable; the fixed point of a permutation reached after applications of \textbf{cds} may depend on the order in which the \textbf{cds} operations were applied.  

This phenomenon on non \textbf{cds}-sortable permutations gives rise to a combinatorial game previously investigated in \cite{ADAMYK} and \cite{JANSEN}. In this game player ONE is assigned a subset of the strategic pile.  The symbol \textbf{CDS}$(\pi, A)$ denotes this game on permutation $\pi$ where $A$ is the subset of the strategic pile assigned to player ONE.  Beginning with ONE, players ONE and TWO take turns successively applying \textbf{cds} to $\pi$ until a fixed point is reached.  If the fixed point is in $A$, ONE wins; otherwise, TWO wins.  Note that the number of moves in the game \textbf{CDS}$(\pi,A)$ is bounded by the length of $\pi$ and thus the game is finite.  By Zermelo's Theorem \cite{ZERMELO}, there exists a winning strategy for some player.  Criteria for determining beforehand which of player ONE or player TWO has a winning strategy in a given instance of the game has yet to be discovered.  For \textbf{cds} non-sortable permutations at large \cite{ADAMYK} discovered a lowerbound on the size of $A$ relative to the size of the strategic pile that ensures that ONE has a winning strategy in the game \textbf{CDS}($\pi,A$).  In \cite{JANSEN} it was discovered for some class of permutations this bound is optimal. 
We broaden the exploration of tight bounds beyond the case explored in \cite{JANSEN}. 

The basic definitions and notations as well as formal definitions for \textbf{cds}, \textbf{cds}-sortability, and strategic pile are introduced in section 2.  Sections 3 through 5 will present our results.  
In section 3, we define a group action on permutations whose strategic pile is said to be \textit{maximal}.  In section 4, we classify permutations with maximal strategic pile based on the number of possible applications of \textbf{cds}.  In section 5, we observe a tighter bound for $A$ in the game \textbf{CDS}$(\pi,A)$ for permutations $\pi$ with maximal strategic pile and a certain number of valid pointer contexts. Finally, in section 6 we discuss directions for future work.

\section{Preliminaries}

In this paper  the \emph{symmetric group} $S_n$ is the group of bijections $\pi : \mathbb Z_n \rightarrow \mathbb Z_n$, under function composition. An element of $S_n$ is said to be a \emph{permutation} of length $n$. All numbers appearing in permutations are always assumed to be coset representatives modulo $n$. In general we choose the smallest \emph{strictly positive} member of the coset as representative, e.g. $n$ is always the representative of the identity element. 
A permutation $\pi \in S_n$ will be denoted
\[
\lbrack \ \pi(1) \ \pi(2) \ \pi(3) \ \ldots \ \pi(n) \ \rbrack
\]
 For convenience 
 the symbol $\lbrack k\rbrack$ denotes the set $\{1, ..., k\}$. A permutation $\pi$ is \emph{sorted} if the entries appear in increasing order, or more precisely, if $\pi$ is the identity function.

\begin{definition} Fix a positive integer $n$.
\begin{enumerate}
    \item{The set $P_n = \{(1,2),\; (2,3),\; \cdots\; (n-1,n)\}$ is said to be the set of $S_n$ \emph{pointers}.}
   \item{For a positive integer $p$, the symbol $p^*$ denotes the pointer $(p,p+1)$.}
    \item{For each positive integer $k$, define $\text{L}(k) = (k-1,k)$ and $\text{R}(k) = (k,k+1)$. $\text{L}(k)$ is said to be the \emph{left pointer} of $k$, and $\text{R}(k)$ is said to be the \emph{right pointer} of $k$.}
    \item{To each permutation $\pi =\lbrack \pi(1)\ \pi(2)\ \cdots \ \pi(n-1)\ \pi(n)\rbrack \in\text{S}_n$ associate its 
 \emph{pointer word} $W(\pi) = \lbrack L(\pi(1))\ R(\pi(1))\ \cdots \ L(\pi(i))\ R(\pi(i))\ \cdots \ L(\pi(n))\ R(\pi(n)) \rbrack$ with entries $(0,1)$ and $(n,n+1)$ removed.} 
\item{For a pointer $p = (j-1,j)$ of permutation $\pi$, the \emph{pointer word, ignoring $p$} is denoted $W(\pi)\setminus p$, and is the result of removing the segment $L(\pi(j)) \; R(\pi(j))$ from $W(\pi)$}
\end{enumerate}
\end{definition}
 Note that the function 
 \[
  W: S_n \rightarrow P_n^{2n-2}
 \]
 is a one-to one function. From an element $\sigma$ in the range of $W$ one can uniquely recover the permutation $\pi\in S_n$ for which $W(\pi) = \sigma$. 
Note that in $W(\pi)$ every pointer appears exactly twice; the pointer $(k, k+1)$ is the right pointer of the element $k$ and the left pointer of the element $k+1$. Thus $W(\pi)$ is an example of a \emph{double occurrence word} \cite{STZ} over the alphabet $P_n$. 

\begin{example}\label{ex:pointers}
Let $\pi$ be the permutation $\lbrack\ 6 \ 3 \ 5 \ 1 \ 2 \ 4 \ \rbrack$. Then the pointer word of $\pi$ is
\[
 W(\pi) = \lbrack (5,6) \quad(2,3) \quad(3,4) \quad (4,5)  \quad (5,6) \quad (1,2)  \quad(1,2)  \quad (2,3) \quad (3,4) \quad (4,5) \rbrack 
\]
 When convenient we write 
\[
\lbrack \ _{(5, 6)}6\quad _{(2, 3)}3_{(3, 4)}\quad _{(4, 5)}5_{(5, 6)}\quad 1_{(1, 2)}\quad _{(1, 2)}2_{(2, 3)}\quad _{(3, 4)}4_{(4, 5)} \ \rbrack
\]
to display both $\pi$ and its pointer word, or
\[
\lbrack \ _{(5, 6)}6_{(6,7)}\quad _{(2, 3)}3_{(3, 4)}\quad _{(4, 5)}5_{(5, 6)}\quad _{(0,1)}1_{(1, 2)}\quad _{(1, 2)}2_{(2, 3)}\quad _{(3, 4)}4_{(4, 5)} \ \rbrack
\]
to display both $\pi$ and its extended pointer word.

For the pointer $2^* = (2,3)$, the pointer word of $\pi$, ignoring $2^*$, is
\[
 W(\pi)\setminus 2^* = \lbrack (5,6) \quad (4,5)  \quad (5,6) \quad (1,2)  \quad(1,2)  \quad (2,3) \quad (3,4) \quad (4,5) \rbrack. 
\]
\end{example}

The notion of an adjacency in a permutation plays a fundamental role in the context directed swap operation \textbf{cds} to be introduced shortly. In preparation for \textbf{cds}, we now define notions to facilitate the exposition.
\begin{definition}
Let $\pi \in S_n$ and $k \in \mathbb Z_n$ be such that $\pi(k) \neq n$. We say that $\pi(k)$ and $\pi(k+1)$ form an \emph{adjacency} if
\[
\pi (k+1) = \pi (k)+1
\]
\end{definition}

When there is an adjacency between $\pi(k)$ and $\pi(k+1)$, then in the pointer word of $\pi$ the right pointer of $\pi(k)$, i.e. $R(\pi(k)) = (\pi(k),\pi(k) + 1) = \pi(k)^*$, appears directly adjacent the left pointer of $\pi(k+1)$, $L(\pi(k+1))$, and $R(\pi(k)) = L(\pi(k+1))$. Thus, when in $\pi$ there is an adjacency between $\pi(k)$ and $\pi(k+1)$, there is a duplication in $W(\pi)$.  
We also say that there is an \emph{adjacency about the pointer} $\pi(k)^* = (\pi(k),\pi(k) + 1)$ of $\pi$.

 This leads us to the following definition:

\begin{definition}\label{def:RPA}
Suppose $\pi \in S_{n}$ has an adjacency about a pointer $r^* = (r,r+1)$. 
\begin{enumerate}
    \item{The \emph{adjacency reduction of $\pi$ at $r^*$}, denoted  $R_{r^*}(\pi)$, is the following element of $S_{n-1}$: Fix $K$ for which $\pi(K) = r + 1$.  Then define
    \[\begin{cases}
      R_{r^*}(\pi)(k) = \pi(k) & \text{ if }\pi(k) \leq r \text{ and } k < K\\
      R_{r^*}(\pi)(k) = \pi(k) - 1 & \text{ if } \pi(k) > r \text{ and } k < K\\
      R_{r^*}(\pi)(k - 1) = \pi(k) & \text{ if } \pi(k) \leq r \text{ and } k > K\\
      R_{r^*}(\pi)(k - 1) = \pi(k) - 1 & \text{ if } \pi(k) > r \text{ and } k > K
      \end{cases}\]} 
    \item{The \emph{reduced pointer word of $\pi$ at $r^*$}, denoted $RW_{r^*}(\pi)$, is defined as follows: Fix $K$ for which $\pi(K) = r+1$. Then for $0\le j\le n-2$
    $RW_{r^*}(\pi)(2j+1) = L(x_j)$ and $RW_{r^*}(\pi)(2j+2) = R(x_j)$ where 
    \[ x_j = \left\{\begin{array}{ll}
                        \pi(j+1) & \mbox{ if } j\neq K-1 \mbox{ and } \pi(j+1)<\pi(K)\\
                        \pi(j+1) - 1 & \mbox{ if } j\neq K-1 \mbox{ and } \pi(j+1)>\pi(K)\\
                        \pi(j+1) & \mbox{otherwise}.
                    \end{array}
                    \right.
    \]
    }
\end{enumerate}
\end{definition}

\begin{example}\label{ex:RPA1} For the permutation $\pi = \lbrack 8\ 5\ 2\ 4\ 6\ 7\ 3\ 1 \rbrack$ consider the pointer $6^* = (6,7)$. In the notation of Definition \ref{def:RPA}, $K = 6$. Now
\[
  W(\pi) = \lbrack 
             (7,8) \ (4,5)\; (5,6)\ (1,2)\; (2,3) \ (3,4)\; (4,5) \ (5,6)\; {\color{red}(6,7)\ (6,7)}\; (7,8) \ (2,3)\; (3,4) \ (1,2)  
            \rbrack
\]
Then we have
\[
  RW_{6^*}(\pi) = \lbrack 
             (6,7) \ (4,5)\; (5,6)\ (1,2)\; (2,3) \ (3,4)\; (4,5) \ (5,6)\; (6,7) \ (2,3)\; (3,4) \ (1,2)  
            \rbrack
\]
\end{example}

The adjacency reduction operation induces a natural bijective mapping from the pointers of $\pi$, excluding pointer $p$, to the pointers of $R_p(\pi)$. 
This correspondence takes the left and right pointers of $\pi(k)$, $k < K - 1$ to the left and right pointers of $R_p(\pi)(k)$, and takes the left and right pointers of $\pi(k)$, $k > K$, to the left and right pointers of $R_p(\pi)(k-1)$. The left pointer of $\pi(K - 1)$ is taken to the left pointer of $R_p(\pi)(K-1)$, and the right pointer of $\pi(K)$ is taken to the right pointer of $R_p(\pi)(K-1)$.

\begin{repexample}{ex:pointers} 
Taking $\pi^{\prime}$ as in example \ref{ex:pointers}, we have adjacencies around the pointers $5^* = (5,6)$ and $3^* = (3,4)$. If we reduce the permutation $\pi^{\prime}$ at $3^* = (3,4)$, the resulting permutation, $R_{(3,4)}(\pi^{\prime})$, is:
\[
[\ 1_{(1, 2)}\quad _{(1, 2)}2_{(2, 3)} \quad _{(3, 4)}4_{(4, 5)}\quad _{(4, 5)}5\quad _{(2, 3)}3_{(3, 4)}\ ]
\]
Of the pointers that are not preserved under the induced correspondence, pointer $4^* = (4,5)$ of $\pi^{\prime}$ is sent to pointer $3^*=(3,4)$ of $R_{(3,4)}(\pi^{\prime})$, and pointer $5^*=(5,6)$ of $\pi^{\prime}$ is sent to pointer $4^*=(4,5)$ of $R_{(3,4)}(\pi^{\prime})$.
\end{repexample}

\begin{definition}
Let $\pi \in S_n$ be given. Let $p$ and $q$ be distinct pointers in $\pi$. We say that $(p, q)$ is a \emph{valid pointer context} if the pointers $p$ and $q$ appear in the order $p\ldots q\ldots p\ldots q$, or in the order $q\dots p\dots q\dots p$, in the pointer word $W(\pi)$ of $\pi$. 
\end{definition}

\begin{repexample}{ex:pointers} 
Taking $\pi$ as in Example \ref{ex:pointers}, the pointers $5^* = (5, 6)$ and $3^* = (3, 4)$ appear in the desired order, so $(5^*=(5,6), 3^*=(3, 4))$ is a valid pointer context. \end{repexample}

\begin{definition}
Let $\pi \in S_n$ be given. Let $p$ and $q$ be distinct pointers in $\pi$ that form a valid pointer context $(p, q)$. Define the \emph{context-directed swap} (or \textbf{cds}) on the permutation $\pi$ at pointers $p,q$ to be the permutation resulting from swapping the two individual blocks of entries of $\pi$ appearing between the pointers $p$ and $q$. Let \textbf{cds}$_{p,q}(\pi) \in S_n$ be the permutation resulting from applying \textbf{cds} for valid pointer context $(p,q)$ to $\pi$.    
\end{definition}

\begin{repexample}{ex:pointers} 
Take permutation $\pi$ again as in Example \ref{ex:pointers}, and take the pointers $5^* = (5, 6)$ and $3^* = (3, 4)$ that form a valid pointer context. The block ``$6 \ 3$" appears between the first $(5, 6) \ (3, 4)$ pointer pair and the block ``$1\ 2$" appears between the second occurrence. Applying \textbf{cds} for the valid pointer context $(5^*,3^*)$ swaps the positions of these two blocks, yielding:
\[
\pi^{\prime} = [\ 1_{(1, 2)}\quad _{(1, 2)}2_{(2, 3)} \quad _{(4, 5)}5_{(5, 6)}\quad _{(5, 6)}6\quad _{(2, 3)}3_{(3, 4)} \quad _{(3, 4)}4_{(4, 5)} \ ]
\]
\end{repexample}

Observe that \textbf{cds} always introduces adjacencies about the pointers $p$ and $q$; in Example \ref{ex:pointers}, $\pi^{\prime} (3)=5, \pi^{\prime} (3+1)=6$, so $5$ and $6$ form an adjacency, and similarly, $\pi^{\prime} (5) = 3, \pi^{\prime} (5+1) = 4$, so $3$ and $4$ form an adjacency.

Since an adjacency about a pointer $p$ guarantees that $p$ will not subsequently appear in a valid pointer context, the adjacency will persist regardless of any subsequent \textbf{cds} operation performed on the permutation.

A \textbf{cds} \emph{fixed point} is a permutation with no valid pointer contexts. Besides the identity permutation $\lbrack 1\ 2\ \cdots \ n\rbrack$, all other \textbf{cds} fixed points in $S_n$ are of the form 
\[
[\ (k+1) \ (k+2) \ (k+3) \ \ldots \ n \ 1\ 2\ 3\ \ldots \ k \ ]
\]
for some $k \in \mathbb Z_n$ \cite{ADAMYK}.

\subsection*{Strategic Pile}

We now introduce a set associated to each permutation that characterizes its sortability and fixed points under \textbf{cds}.

Fix arbitrary $\pi \in S_n$, say $\pi = [\ a_1 \ a_2 \ a_3 \ \ldots \ a_n \ ]$. For the purposes of the upcoming definition of the strategic pile, view all numbers as elements of $\mathbb Z$, and let parentheses denote the usual cycle notation for permutations. Define
\begin{align*}
   X_n &:= (0\ 1\ 2\ \ldots \ n)\\
   Y_{\pi} &:= (a_n \ a_{n-1} \ a_{n - 2} \ \ldots \ a_1 \ 0)\\
   C_{\pi} &:= Y_{\pi} \circ X_{n},
\end{align*}
the function composition of permutations $Y_{\pi}$ and $X_n$.

\example\label{ex:strpile}
{
 For $\pi = \lbrack 8\ 1\ 5\ 2\ 4\ 3\ 7\ 6 \rbrack$ we find:
 \begin{align*}
   X_8 &:= (0\ 1\ 2\ 3\ 4\ 5\ 6\ 7\ 8)\\
   Y_{\pi} &:= (6 \ 7 \ 3 \ 4\ 2\ 5\ 1\ 8\ 0)\\
   C_{\pi} &:= Y_{\pi} \circ X_{8}\\
     & = (0\ 8\ 6\ 3\ 2\ 4\ 1\ 5\ 7)
\end{align*}
}

\definition{Consider a permutation $\pi \in S_n$. The \emph{strategic pile} of $\pi$, denoted $\text{SP}(\pi)$, is the set of numbers appearing after $n$ and before $0$ in the cycle of $C_\pi$ containing $0$ and $n$. If $0$ and $n$ do not appear in the same cycle, then $\text{SP}(\pi) = \emptyset$.}

\begin{repexample}{ex:strpile}
For $\pi = \lbrack 8\ 1\ 5\ 2\ 4\ 3\ 7\ 6 \rbrack$ the strategic pile is $\text{SP}(\pi) = \{6,\; 3,\; 2,\; 4,\; 1,\; 5,\; 7\}$
\end{repexample}

The strategic pile completely characterizes when a permutation is sortable under \textbf{cds}.

\begin{theorem}[\cite{ADAMYK}, Theorem 2.18]\label{thm:sortable}
A permutation $\pi \in S_n$ is \textbf{cds} sortable if and only if $\text{SP}(\pi) = \emptyset$.
\end{theorem}

In the case that a permutation is not sortable, the strategic pile computes precisely which fixed points are reachable under \textbf{cds}.

\begin{theorem}[\cite{ADAMYK}, Theorem 2.22]\label{thm:fixedpt}
Let $\pi \in S_n$ be given. Then
\[\lbrack k\; k+ 1\; ... \; n \; 1\;2 \; ... \; k - 1\rbrack\]
is a fixed point of $\pi$ if and only if $k - 1\in \text{SP}(\pi)$.
\end{theorem}

In this paper we focus on permutations which have the largest possible strategic pile.

\begin{theorem}[\cite{GAETZ} Lemma 3.1, Corollary 3.2]\label{thm:strpilebounds}
Let $\pi \in S_{n}$. If $n$ is even, then $|\text{SP}(\pi)| \leq n - 1$. If $n$ is odd, then $|SP(\pi)| \leq n - 2$.
\end{theorem}

We now have a way to make precise what we mean by a permutation having the largest possible strategic pile.

\definition\label{def:maxpile}{Let $\pi \in S_n$. We say that $\pi$ has a \emph{maximal strategic pile} if $n$ is even and $|\text{SP}(\pi)| = n - 1$, or $n$ is odd and $|\text{SP}(\pi)| = n - 2$.}

\begin{repexample}{ex:strpile}
$\pi = \lbrack 8\ 1\ 5\ 2\ 4\ 3\ 7\ 6 \rbrack$ has a maximal strategic pile. The fixed points reachable by repeated applications of \textbf{cds} are $\lbrack 8\ 1\ 2\ 3\ 4\ 5\ 6\ 7\rbrack$,
$\lbrack 7\ 8\ 1\ 2\ 3\ 4\ 5\ 6 \rbrack$,
$\lbrack 6\ 7\ 8\ 1\ 2\ 3\ 4\ 5 \rbrack$,
$\lbrack 5\ 6\ 7\ 8\ 1\ 2\ 3\ 4 \rbrack$,
$\lbrack 4\ 5\ 6\ 7\ 8\ 1\ 2\ 3 \rbrack$,
$\lbrack 3\ 4\ 5\ 6\ 7\ 8\ 1\ 2 \rbrack$, and 
$\lbrack 2\ 3\ 4\ 5\ 6\ 7\ 8\ 1\rbrack$.
\end{repexample}

The following proposition collects several properties of permutations with maximal strategic pile.

\begin{proposition}\label{prop:maxpileprops} If $\pi \in S_n$ have maximal strategic pile, then the following are true.
\begin{enumerate}
    \item Let $n$ be even. If $\pi(k) = n$, then $\pi(k + 1) = 1$.
    \item If $n$ is even, then $\pi$ has no adjacencies. 
    \item If $n$ is odd, $\pi$ has precisely one adjacency.
    \item Let $n$ be even. An application of \textbf{cds} to $\pi$ introduces exactly two adjacencies, and removes exactly two elements from the strategic pile. 
    \item Let $n$ be even. If one performs \textbf{cds} on $\pi$, after reducing the resulting adjacencies we have a new permutation of maximal strategic pile in $S_{n - 2}$.
\end{enumerate}
\end{proposition}

\begin{proof}
We first prove item (1). If $\pi\in S_{n}$ has a strategic pile of size $n-1$, then each number in $[n-1]$ must appear between $n$ and $0$ in a cycle of $C_\pi$. Thus, $C_\pi(0)=n$. By the definition of $C_\pi$, we must have that $\pi^{-1}(0)=\pi^{-1}(n)-1$, as desired.
Items (2) and (3) are from \cite{GAETZ}, Lemma 3.11.
Moreover, (5) is a corollary of (4), so we prove (4). Each application of \textbf{cds} introduces at least two adjacencies, and thus removes at least two elements from the strategic pile (\cite{ADAMYK}, Lemma 2.7). However, no more than two elements can be removed from the strategic pile with a single \textbf{cds} move (\cite{ADAMYK}, Corollary 2.16).
\end{proof}

By the preceding proposition, to study permutations with maximal strategic pile, it suffices to study such permutations with even length: Permutations of odd length with maximal strategic pile can be reduced to permutations with even length and maximal strategic pile.

Furthermore, for a permutation $\pi \in S_{2n}$ with maximal strategic pile, entry $2n$ always occurs directly to the left of $1$. Thus we can contract these entries in a way analogous to reductions of adjacencies.

\definition{To each $\pi \in S_{2n}$ with maximal strategic pile, associate a permutation $\pi' \in S_{2n - 1}$, defined as follows: If $m$ is such that $\pi(m) = 2n$, then for all $k < m$ define
\[\pi'(k) = \pi(k)\]
and for all $m \leq k < 2n - 1$ define
\[\pi'(k) = \pi(k + 1)\]
and assign $\pi'$ the same pointers as the corresponding elements in $\pi$, except $1$ has the left pointer $(2n - 1,1)$, and $2n-1$ has the right pointer $(2n-1,1)$. 

$M_{2n,k}$ is the set of all such contractions on permutations of length $2n$ with maximal strategic pile and $k$ valid pointer contexts.
}

\example\label{ex:reduct}{The permutation
$\pi = [2\; 4\; 6\; 1\; 3\; 5]$ 
in $S_6$ has a maximum strategic pile. Also, $\pi$ has $k = \binom{2\cdot 3 - 1}{2}$ pointer pairs that constitute valid pointer contexts. Its contraction to a permutation in $M_{6,k}$ is
$\lbrack 2\;4\;1\;3\;5\rbrack$.
}

\section{A Group Action on $M_{2n,k}$}

A goal of this section is to count, fixing $k$ and $n$, the number of permutations that are of length $2n$ have maximum strategic pile, and exactly $k$ valid pointer contexts. Note that as a consequence of 1) in Proposition \ref{prop:maxpileprops}, $\vert M_{2n,k}\vert$ is the number of permutations $S_{2n}$ with maximum strategic pile and $k$ pairs of pointers that constitute valid pointer contexts. The enumeration method relies on a group action on permutations which preserves membership to the set $M_{2n,k}$, and an application of the orbit-stabilizer theorem.

\definition{Define the map $\phi: \mathbb Z_n \times \mathbb Z_n \times S_n \rightarrow S_n$ by
\[\phi(a,b,\pi)(x) = \pi(x - a) + b\]
where $a,b \in \mathbb Z_n$ and $\pi \in S_n$.
}

For $\pi \in S_n$, this operation has the effect of applying a cyclic shift by $a$ to $\pi$, followed by adding $b$ to each element in $\pi$.

\example\label{ex:phi}{
If we consider
$\pi = [5\; 4 \; 1\; 3 \; 2]$
in $S_5$, then
\[\phi(2,3,\pi) = [3 + 3\; 2 + 3\; 5 + 3\; 4 + 3\; 1 + 3] = [1 \; 5\; 3\; 2\; 4]\]
}

As is exhibited in Example 3.2, the function $\phi$ is composed of two elementary functions: (a) \emph{cyclic shift by } $a \in \mathbb Z_n$ which is defined by
\[C_a: S_n \rightarrow S_n\]
\[\pi \mapsto \pi(x - a)\]
and (b) \emph{translation by} $b \in \mathbb Z_n$ which is defined by
\[T_b: S_n \rightarrow S_n\]
\[\pi \mapsto \pi(x) + b\]
Then we have that
\[\phi(a,b,\pi) = C_a(T_b(\pi)) = T_b(C_a(\pi))\]
for all $a,b \in \mathbb Z_n$, $\pi \in S_n$. Furthermore, note that
\[C_a \circ C_b = C_{a + b}\]
\[T_a \circ T_b = T_{a + b}\]
for all $a,b \in \mathbb Z_n$.

We now show that the action of $\mathbb Z_{2n - 1} \times \mathbb Z_{2n - 1}$ on any element of $M_{2n,k}$ gives an element of $M_{2n,k}$.

\begin{theorem}$\phi$ restricted to $\mathbb Z_{2n - 1} \times \mathbb Z_{2n - 1} \times M_{2n,k}$ has image a subset of $M_{2n,k}$.
\end{theorem}

\begin{proof}
Let $\pi \in M_{2n,k}$. It suffices to show that 
\[\phi(1,0,\pi) \in M_{2n,k}\]
\[\phi(0,1,\pi) \in M_{2n,k}\]
Let $\pi' \in S_{2n}$ be the permutation for which $\pi$ is a reduction.

First $\phi(1,0,\pi)$ is a reduction of a permutation with the same number of moves as $\pi'$. This is seen by noting if $\sigma$ is the permutation for which $\phi(1,0,\pi)$ is a reduction, then the order of the pointers is the same in $\sigma$ as it is in $\pi'$, except the second to last and last pointers in $\pi'$ appear first and second respectively in $\sigma$. Suppose $(p,q)$ are a valid pointer context in $\pi'$. Then they appear in the order $\ldots p \ldots q \ldots p \ldots q$ in $\pi'$. Suppose then $p$ nor $q$ was one of the last two pointers in $\pi'$, in which case the pointers appear in the order $\ldots p \ldots q \ldots p \ldots q$ in $\sigma$, and thus form a valid pointer context. Otherwise, if $q$ was one of the last two pointers, then the list of pointers is of the form $\ldots q \ldots p \ldots q \ldots p$ in $\sigma$ and thus $(p,q)$ is a valid pointer context. Once again if $p$ and $q$ were the last two pointers in $\pi'$ then they appear as $pq\ldots p\ldots q$ in $\sigma$ and once again form a valid pointer context.

The same is true for $\phi(0,1,\pi)$ since if $...p...q...p...q...$ is a valid pointer context in $\pi'$, then the pointer word after translation has the form $...(p + 1)...(q + 1)...(p+1)...(q+1)...$, containing a valid pointer context.

For permutations with even length and maximal strategic pile we have that $C_{\pi'}(0) = 2n$ and $C_{\pi}(2n) = \pi'(2n)$, since $2n$ appears directly before $1$ in such a permutation. We have that the strategic pile of the permutation that reduces to $\phi(1,0,\pi)$ is a cyclic shift by one of the original permutation, and likewise the strategic pile of the permutation that reduces to $\phi(0,1,\pi)$ is a translation by $1$ of the original permutation. This implies the theorem.
\end{proof}

From now on, let $\phi'$ denote the restriction of $\phi$ to $\mathbb Z_{2n - 1} \times \mathbb Z_{2n - 1} \times M_{2n,k}$.

\begin{corollary}$\phi'$ is a $\mathbb Z_{2n - 1} \times \mathbb Z_{2n - 1} - $action on $M_{2n,k}$.
\end{corollary}

\begin{proof}
Firstly, the identity $(0,0) \in \mathbb Z_{2n - 1} \times \mathbb Z_{2n - 1}$ preserves any element in $M_{2n,k}$. Let $\pi \in M_{2n,k}$.
\[\phi(0,0,\pi)(x) = \pi(x - 0) + 0 = \pi(x)\]
Thus, $\phi(0,0,\pi) = \pi$. Secondly, if $(a,b),(c,d) \in \mathbb Z_{2n - 1} \times \mathbb Z_{2n - 1}$
\[\phi(a + c, b + d,\pi)(x) = \pi(x - (a + c)) + b + d\]
\[= [\pi((x - c) - a) + b] + d = \phi(c,d,\phi(a,b,\pi))(x)\]
which proves the statement.
\end{proof}

Notice that 
\[\pi(x + 1) - \pi(x) = T_b(\pi)(x + 1) - T_b(\pi)(x)\]
and so, intuitively, two permutations are translations of one another if the differences between their elements are the same. This motivates the following definition.

\definition{
Let permutation $\pi$ be an element of $S_n$. The \emph{difference sequence} of $\pi$ is the $n-$tuple $D_\pi \in \mathbb Z_n^n := \prod_{i = 1}^n \mathbb Z_n$ where for $1 \leq k \leq n - 1$, the $k$th component of $D_\pi$ is defined as follows:
\begin{equation}
  D_\pi(k) = \left\{
                           \begin{array}{ll}
                                   \pi(k + 1) - \pi(k) & \mbox{ if } 1\leq k<n\\
                                   \pi(1) - \pi(n)       & \mbox{ otherwise}
                           \end{array}
                  \right.
\end{equation}
In the case that there exists a $0 < p < n$ such that
\begin{equation}
D_\pi(k + p) = D_\pi(k)
\end{equation}
$\pi$ is said to have a \emph{periodic difference sequence with period} $p$ if $p$ is the smallest integer such that $D_\pi$ ($3.2$) is satisfied. If $D_\pi$ is periodic with period $1$, then $\pi$ is said to have \emph{constant difference sequence}.
}

\begin{example}\label{ex:difference}
Consider $[\pi = [2\;4\;3\;8\;1\;9\;5\;7\;6]$ in $\textsf{S}_9$. Then $D_\pi = (2,8,5,2,8,5,2,8,5)$ which is a periodic difference sequence with period 3.
\end{example}

Given the difference sequence of a permutation $\pi \in M_{2n,k}$, its stabilizer under $\phi'$, denote $\text{Stab}(\pi)$, can be computed directly.

\begin{remark}
If $b \in \mathbb Z_n$ and $\pi \in S_n$, then $D_{T_b(\pi)} = D_\pi$. This is seen by computing for each $x \in \mathbb Z_n$ that 
$D_{T_b(\pi)}(x) = (\pi(x+1) + b) - (\pi(x) + b) = \pi(x + 1) - \pi(x) = D_\pi(x)$,
as desired.
\end{remark}

\begin{lemma}
Let $\pi \in S_n$ have periodic difference sequence with period $p$, and $a,b \in \mathbb Z_n$, $a$ and $b$ both not the identity, such that $C_a \circ T_b \circ \pi = \pi$. Then $A$ is an integer multiple of $p$.
\end{lemma}

\begin{proof}
Since translations preserve the difference sequence of a permutation, $D_{C_a(\pi)} = D_\pi$. A cyclic shift by $a$ to a permutation $\sigma \in S_n$ applies a cyclic shift by $a$ to $D_\sigma$. This is seen as follows:
\[D_\sigma(k) = \sigma(k + 1) - \sigma(k)\]
and so
\[D_{C_a(\sigma)}(k) = C_a(\sigma)(k + 1) - C_a(\sigma)(k)\]
\[= \sigma(k + 1 - a) - \sigma(k - a) = D_{\sigma}(k - a)\]
If $p$ does not divide $a$, then $a = kp + r$ for some nonnegative integers $k,r$ with $0 < r < p$. By the periodicity of the difference sequence the cyclic shift applied to the difference sequence by $a$ is equivalent to a cyclic shift by $r$, which contradicts the fact that $p$ is the smallest positive integer such that ($3.1$) holds for $D_\pi$.
\end{proof}

\begin{theorem}If $\pi \in M_{2n,k}$ then $\text{Stab}(\pi)$ under the action $\phi'$ is cyclic. Moreover, if $D_\pi$ is periodic with period $p$, then $\text{Stab}(\pi)$ is generated by $(p,\pi(p) - \pi(0))$.
\end{theorem}

\begin{proof}
First we show that $(p,\pi(p) - \pi(0)) \in \text{Stab}(\pi)$. If $x \in \mathbb Z_{2n - 1}$ we have
\[\phi(p,\pi(p) - \pi(0),\pi)(x) = \pi(x - p) + \pi(p) - \pi(0)\]
Since $D_\pi$ is periodic with period $p$, $\pi(x) - \pi(x - p) = \pi(p) - \pi(0)$ and thus
\[= \pi(x)\]
as desired.

Let $a,b \in \mathbb Z_{2n - 1}$ such that $C_a \circ T_b \circ \pi = \pi$. We claim that either $a = b = 0$ or $a \neq 0$ and $b \neq 0$. First suppose that $a = 0$ and $b \neq 0$. Then $T_b(\pi) = \pi$, which is a contradiction since $b \neq 0$. Now if $a \neq 0$ and $b = 0$, we have that $C_a(\pi) = \pi$, a contradiction since $a \neq 0$.

Now suppose that $a,b \in \mathbb Z_{2n - 1}$ are such that $a \neq 0$, $b \neq 0$ and $C_a \circ T_b \circ \pi = \pi$. There exists $q \in \mathbb Z_{2n - 1}$ such that $C_p \circ T_q \circ \pi = \pi$. Then we have that, for any $k$ a positive integer.
\[C_{kp} \circ T_{kq} \circ C_a \circ T_b \circ \pi = C_{kp} \circ C_a \circ T_{kq} \circ T_b \circ \pi\]
\[= C_{kp + a} \circ T_{kq + b} \circ \pi = \pi\]
Then by Lemma $3.7$, we could have taken $k$ such that $kp + a = 0$, and thus $kq + b = 0$, proving the theorem.
\end{proof}

\begin{corollary}$\text{Stab}(\pi)$ is trivial if and only if $D_\pi$ is not periodic.
\end{corollary}

\begin{proof}
Apply the proof of Lemma $3.7$
\end{proof}

\begin{corollary}The orbit of $\pi$ under $\phi'$, denoted $\mathcal{O}_\pi$, has order $(2n - 1)p$ if $D_\pi$ is periodic with period $p$. $\mathcal{O}_\pi$ is of order $(2n - 1)^2$ otherwise.
\end{corollary}

\begin{proof}
This is a consequence of the orbit-stabilizer theorem. First let $D_\pi$ be periodic with period $p$. The order of $(p,\pi(p) - \pi(0)) \in \mathbb Z_{2n - 1} \times \mathbb Z_{2n - 1}$ is $\frac{2n - 1}{p}$ since $p$ has order $\frac{2n - 1}{p}$ in $\mathbb Z_{2n - 1}$, and
\[\frac{2n - 1}{p}(\pi(p) - \pi(0)) = \sum_{i = 1}^{(2n-1)/p}[\pi(ip) - \pi((i - 1)p)] = \pi\left(\frac{2n - 1}{p}p\right) - \pi(0) = 0\]
Thus $\pi(p) - \pi(0)$ has order dividing $\frac{2n - 1}{p}$ in $\mathbb Z_{2n - 1}$. Hence $|\text{Stab}(\pi)| = \frac{2n - 1}{p}$. Therefore by the orbit-stabilizer theorem
\[|\mathcal{O}_\pi| = \frac{|\mathbb Z_{2n - 1} \times \mathbb Z_{2n - 1}|}{|\text{Stab}(\pi)|}\]
\[= \frac{(2n - 1)^2}{\frac{2n - 1}{p}} = (2n - 1)p\]
as desired. For the case when $D_\pi$ is not periodic, $|\text{Stab}(\pi)| = 1$ and thus
\[|\mathcal{O}_\pi| = \frac{(2n - 1)^2}{1}\]
as desired.
\end{proof}

\subsection*{Periodic Difference Sequence Characterization}
We return to difference sequences to characterize permutations permutations with nontrivial stabilizer. In particular we will characterize permutations with periodic difference sequences and further characterize permutations with both periodic difference sequence and maximal strategic pile. These results aid in counting the number of permutations with maximal strategic pile and periodic difference sequence with specified period. 

\begin{lemma}\label{elem_diff}
For a permutation $\pi \in S_n$ with difference sequence $D_{\pi}$, for all $i,
j \in [n]$
\[\pi(i+j) - \pi(i) = \sum_{k=0}^{j-1} D_{\pi}(i+k)\]
\end{lemma}
\begin{proof}
\begin{align*}
\pi(i+j) - \pi(i) &= \pi(i + j) - \pi(i) + \sum_{k=1}^{j - 1} \pi(i + k) - \sum_{k=1}^{j - 1} \pi(i + k)\\
&= \sum_{k=1}^{j} \pi(i + k) - \sum_{k=0}^{j - 1} \pi(i + k)\\
&= \sum_{k=0}^{j-1}\pi(i+k+1) - \pi(i+k)\\
&= \sum_{k=0}^{j-1}D_{\pi}(i+k)
\end{align*}
\end{proof}

\begin{lemma}\label{diff_criteria}
An $n-$tuple $E \in \mathbb Z_{n}^n$ is the difference sequence of a permutation in $S_{n}$ if and only if $\sum_{i=1}^{n} E(i) = 0$
and there do not exist $i,j \in [n], j < n$ such that
$\sum_{k=0}^{j-1} E(i+k) = 0$
\end{lemma}
\begin{proof}
Suppose $E$ is the difference sequence of some permutation $\pi \in S_n$. Fix arbitrary $i, j \in [n]$. By Lemma \ref{elem_diff},
\[\sum_{k=1}^{n} E(k) = \sum_{k=0}^{n-1} E(1+k) = \pi(1) - \pi(1+n) =  0 \]
Suppose towards a contradiction that there are $i,j \in [n], j < n$ such that
\[\sum_{k=0}^{j-1} E(i+k) = 0 \]
Then by Lemma \ref{elem_diff}, $\pi(i+j) - \pi(i) = 0$, so $\pi(i) = \pi(i+j)$ but because $j \neq n, i \neq i+j$, contradicting the injectivity of $\pi$.

Now let $E \in \mathbb Z_{n}^n$ such that
$\sum_{i=1}^{n} E(i) = 0$
and assume that there are no $i,j \in [n], j < n$ such that
$\sum_{k=0}^{j-1} E(i+k) = 0$.
Fix arbitrary $a_1 \in \mathbb Z_n$ and define $\alpha: \mathbb Z_n \rightarrow \mathbb Z_n$ by:
\begin{align*}
\alpha(1) &\mapsto a_1\\
\alpha(i) &\mapsto \alpha(i-1) + E(i-1) & \text{for } i \in [n], i > 1
\end{align*}
To show that $\alpha$ is a permutation, it is enough to show that $\alpha$ is injective. 
By construction, we have that for all $i, j \in [n]$,
\[\alpha(i) + \sum_{k=0}^{j-1} E(i+k) = \alpha(i+j) \]
Assume for a contradiction that there exist $i, j \in \mathbb Z_n, i \neq j$ such that $\alpha(i) = \alpha(j)$. Then
\[ \alpha(i) - \alpha(j) = \sum_{k=i}^{j-1}E(k) = 0\]
contradicting the choice of $E$.
\end{proof}

\begin{remark}
The permutation $\alpha$ constructed above is not unique, as any element of $\mathbb Z_n$ can be chosen to be $a_1$. In particular, $\alpha$ constructed by a given choice of $a_1$ is some translation away from every other permutation with the same difference sequence.
\end{remark}

For a permutation $\pi \in S_n$, let $\pi_k$ denote the function
$\pi_k: \mathbb Z_k \rightarrow \mathbb Z_k$ satisfying
$\pi_k(i) \equiv \pi(i) \bmod k$ for all $i \in \mathbb{Z}_k$.

\begin{lemma}
\label{pib}
Let $\pi \in S_{n}$ and $p \mid n$, and suppose the difference sequence of $\pi$
is periodic with a period dividing $p$. Then for all $1 \le i \le n,
\pi(i) \equiv \pi(i+p) \bmod p$.
\end{lemma}
\begin{proof}
Recall that $\sum_{i=1}^{n} D_{\pi}(i) = 0$. By periodicity of the $D_{\pi}$,
for all $i, j \in \mathbb [n]$, we have
\[\sum_{k=i}^{i+p-1}D_{\pi}(k) = \sum_{k=j}^{j+p-1}D_{\pi}(k)\]
Let $s$ denote this sum of any $p$ consecutive elements of $D_{\pi}$. Necessarily,
\[\sum_{k=1}^{n}D_{\pi}(k) = \frac{n}{p} \cdot s \equiv 0 \bmod n\]
Therefore, $p \mid s$, so for all $i \in \mathbb Z_n, \pi(i) \equiv
\pi(i+p) \bmod p$.
\end{proof}

\begin{lemma}
\label{pib_diff}
Let $\pi, p$ be defined as in Lemma \ref{pib}. Then $\pi_b$ is a permutation in $S_p$
\end{lemma}
\begin{proof}
We show $D_{\pi_p}$ satisfies the criteria of Lemma \ref{diff_criteria}.
Let $s$ be defined as in Lemma \ref{pib}.
$\sum D_{\pi_p} \equiv s \bmod p$, and as seen in Lemma \ref{pib}, $p \mid s$,
so $\sum D_{\pi_p} \equiv 0 \bmod{p}$.

Assume that on the contrary there exist  $i, t \in [p], t < p$
such that $\sum_{k=i}^{i+t-1} D_{\pi_p}(k) \equiv 0 \bmod p$. Let $j = i + t$.
Note that because $t \neq p, i \not\equiv j$. The assumption is
equivalently, $\pi_p(i) \equiv \pi_p(j) \bmod p$.
For all $i \in \mathbb Z_n$, there are $\frac{n}{p}$ elements of $\mathbb
Z_n$ (including $i$ itself) equivalent $\phantom{}\bmod p$ to $i$.
By Lemma \ref{pib}, for each $f \in \mathbb Z_{\frac{n}{p}}$,
$\pi(i) \equiv \pi(i+fp)
\bmod p$, so because $\pi$ is a bijection, there must exist
$f \in \mathbb Z_{\frac{n}{p}}$ such that $j = i + fp$. But then
$j \equiv i + fp$, a contradiction.
\end{proof}

\begin{repexample}{ex:difference}
Consider the permutation $\pi\in\textsf{S}_9$ with periodic difference sequence:
\begin{align*}
\pi &= [2\;4\;3\;8\;1\;9\;5\;7\;6]\\
D_\pi &= (2,8,5,2,8,5,2,8,5)
\end{align*}
$D_\pi$ has period 3. $\pi_3 = [2\;1\;3\;2\;1\;3\;2\;1\;3]$ which is $\frac{9}{3} = 3$ copies of $[2\;1\;3]$.
\end{repexample}

We now describe what information is needed to construct a permutation with periodic difference sequence.

\begin{theorem}
    \label{timewheel}
    Let $n \in \mathbb N, \ p \mid n$. A triple
    \[
    (\varphi,
    \; \mathcal R,
    \; k
    )
    \in
    S_p \times
    \left(
    \prod_{i=1}^{p} \mathbb \{0\} \cup \left[\frac{n}{p} - 1\right]\right) \times
    \mathbb Z_{\frac{n}{p}}
    \]
    defines a permutation $\pi \in S_n$ with periodic difference sequence (having period dividing $p$) by
    \[\pi(i) = \mathcal R_{i \bmod p} \cdot p + \varphi(i \bmod p)) + k p \left \lfloor{\frac{i-1}{p}} \right \rfloor\]
    if $\text{gcd}(k,\frac{n}{p}) = 1$.
\end{theorem}

\begin{proof}
    Let $i,j \in \mathbb{Z}_n, \, i \not\equiv j \bmod n$. Without loss
    of generality, let $i > j$. We have
    \begin{align*}
        \pi(i) - \pi(j)
        &= p \cdot \left(\mathcal R_{i \bmod p} - \mathcal R_{j \bmod p}
        + k \cdot \left(\left \lfloor{\frac{i-1}{p}} \right \rfloor
        - \left \lfloor{\frac{j-1}{p}} \right \rfloor\right)\right)
        + \varphi(i \bmod p) - \varphi(j \bmod p)
    \end{align*}
    In the case that $i \not\equiv j \bmod p$,  $0 < \varphi(i \bmod p) - \varphi(j \bmod p) < p$ and so $\pi(i) - \pi(j) \not\equiv 0 \bmod p$. Since $p \mid n$, $\pi(i) - \pi(j) \not\equiv 0 \bmod n$. Thus $\pi$ is a permutation.
    
    In the case that $i \equiv j \bmod p$, we have
    \begin{align*}
        \pi(i) - \pi(j)
        &= k p \cdot \left(\left \lfloor{\frac{i-1}{p}} \right \rfloor
        - \left \lfloor{\frac{j-1}{p}} \right \rfloor\right)\\
    \end{align*}   
    Since $i \neq j$ and $i \equiv j \bmod p$, $\left \lfloor{\frac{i-1}{p}} \right \rfloor
    \neq \left \lfloor{\frac{j-1}{p}} \right \rfloor$. Therefore,
    $0 < \left \lfloor{\frac{i-1}{p}} \right \rfloor
    - \left \lfloor{\frac{j-1}{p}} \right \rfloor < \frac{n}{p} < n$ and since
    gcd$(k, \frac{n}{p}) = 1$ as well, $\pi(i) - \pi(j) \not\equiv 0 \bmod n$.
    
    Now we show $\pi$ has periodic difference sequence. Let $i \in \mathbb{Z}_n$.
    \begin{align*}
        &[\pi(i+1) - \pi(i)] - [\pi(i + 1 + p) - \pi(i + p)]\\
        &=
        [\pi(i + p) - \pi(i)] - [\pi(i+1) - \pi(i+p+1)]\\
        &=
        (\mathcal{R}_{(i + p)\bmod p} - \mathcal{R}_{i \bmod p}) \cdot p + \varphi((i + p)\bmod p) - \varphi(i\bmod p) + kp\left(\left\lfloor\frac{i + p - 1}{p}\right\rfloor - \left\lfloor \frac{i - 1}{p} \right\rfloor \right)\\ &
        - \left((\mathcal{R}_{(i + 1)\bmod p} - \mathcal{R}_{(i + p + 1) \bmod p}) \cdot p + \varphi((i + 1)\bmod p) - \varphi((i + p + 1)\bmod p) + kp\left(\left\lfloor\frac{i}{p}\right\rfloor - \left\lfloor \frac{i + p}{p} \right\rfloor \right)\right)\\
        &=
        kp\left(
        \left \lfloor{\frac{i+ p -1}{p}} \right \rfloor
        -
        \left \lfloor \frac{i - 1}{p} \right \rfloor
        +
        \left \lfloor \frac{i}{p} \right \rfloor
        -
        \left \lfloor \frac{i + p}{p} \right \rfloor\right)\\
        &=
        kp\left(
        \left \lfloor{\frac{i - 1}{p}} \right \rfloor
        -
        \left \lfloor \frac{i - 1}{p} \right \rfloor + 1
        +
        \left \lfloor \frac{i}{p} \right \rfloor
        -
        \left \lfloor \frac{i}{p} \right \rfloor - 1\right)\\
        &=
        0
    \end{align*}
\end{proof}

The converse is actually true. That is, the information in \ref{timewheel} is precisely the amount of information needed to construct a permutation with periodic difference sequence. 

\begin{theorem}
    $\pi \in S_n$ has a periodic difference sequence with period dividing $p$ (for $p \mid n$) if and only if there exists a triple:
    $$
    (\varphi,
    \; \mathcal R,
    \; k
    )
    \in
    S_p \times
    \left(\prod_{i=1}^{p} \mathbb \{0\} \cup \left[\frac{n}{p} - 1\right]\right) \times
    \mathbb Z_{\frac{n}{p}}
    $$
    such that for all $i \in \mathbb Z_n, \; \pi(i) = \mathcal R_{i \bmod p} \cdot p + \varphi(i \bmod p)) + k p \left \lfloor{\frac{i-1}{p}} \right \rfloor$
    and $gcd(k, \frac{n}{p}) = 1$. In fact $\pi$ is uniquely defined by such a triple.
\end{theorem}
\begin{proof}
    The `if' portion of this statement was shown in Theorem \ref{timewheel}.
    We now prove the `only if' portion. Let $\pi \in S_n$ have periodic difference sequence with period dividing $p$.
    We recover $\varphi \in S_p$ from the first $p$ elements of $\pi$ taken
    mod $p$: $\varphi(i) := \pi(i) \bmod p$. By Lemma \ref{pib_diff}, $\varphi$ defined
    in this way is indeed a permutation in $S_p$.
    
    Define $\mathcal{R}$ for $1 \leq i \leq p$ by
    \begin{align*}
    \mathcal R_{i} &:= \frac{\pi(i) - (\pi(i) \bmod p)}{p}
    \end{align*}
    Let $k \in \mathbb Z_{\frac{n}{p}}$ be defined by
    \begin{align*}
    k := \frac{\pi(p + 1) - \pi(1)}{p}
    \end{align*}
    By Lemma \ref{pib}, $k$ defined in this manner is indeed an integer. Also we have that
    \[\pi(1 + mp) - \pi(1) = kpm\]
    for any integer $0 \leq m < \frac{n}{p}$. If $k$ and $\frac{n}{p}$ were not coprime, then we'd have that $kp$ and $n$ are not coprime, in which case there is an $m$ in the specified range so that
    \[= 0 \bmod n\]
    which contradicts the fact that $\pi$ is a permutation. We now show that we can recover $\pi$ with the formula in the statement of the theorem.
    \begin{align*}
    & \mathcal R_{i \bmod p} \cdot p + \varphi(i \bmod p) +
    k p \left \lfloor{\frac{i-1}{p}} \right \rfloor\\
    &= \pi(i \bmod p) + (\pi(p + 1) - \pi(1)) \left \lfloor{\frac{i-1}{p}}
    \right \rfloor\\
    &= \pi(i \bmod p) + (\pi(i) - \pi(i-p)) \left \lfloor{\frac{i-1}{p}} \right
    \rfloor\\
    &= \pi(i \bmod p) + \left( \sum_{\ell=i-p}^{i-1} \pi(\ell + 1) - \pi( \ell )
    \right) \left \lfloor{\frac{i-1}{p}} \right \rfloor\\
    &= \pi(i \bmod p) + \sum_{j = 1}^{\left \lfloor{\frac{i - 1}{p}}\right \rfloor} \left(
    \sum_{\ell=i \bmod p + jp}^{i \bmod p + (j+1)p - 1} \pi(\ell + 1) - \pi( \ell )
    \right)
    & \text{by Lemma \ref{pib}}\\
    &= \pi(i \bmod p) + \sum_{\ell=i \bmod p}^{i-1} \pi(\ell + 1) - \pi( \ell )\\
    &= \pi(i \bmod p) + \pi(i) - \pi(i \bmod p) \\
    &= \pi(i)
    \end{align*}
\end{proof}

\subsection*{The Strategic Pile of Permutations with Periodic Difference Sequences}
\begin{theorem} A permutation $\pi \in S_{2n-1}$ with difference sequence having period $p$ has maximal strategic pile if and only if the following conditions hold:
\begin{enumerate}
    \item Let $\varphi$ be the permutation in $S_p$ gotten from reducing the first $p$ elements of $\pi$ $\bmod p$. The unreduced counterpart of $\varphi$ in $S_{p+1}$ (relinquishing the identification of $1$ and $p + 1$) has maximal strategic pile.
    \item If $K=\frac{1}{p}(\pi(p+1)-\pi(1))$, then the order of $K-1\bmod \frac{2n-1}{p}$ is $\frac{2n-1}{p}$
\end{enumerate}
\end{theorem}
\begin{proof}

Let $\rho$ be the permutation achieved by reducing the first $p$ elements of $\pi\bmod p$. We will first show that $\rho$ must have maximal strategic pile. Let $S$ be the size of the strategic pile of $\rho$. By the definition of $C_\pi$ and the periodicity of $\pi\bmod p$, if one reduces the numbers listed in $C_\pi\bmod p$, the resulting list of numbers after $2n$ and before $0$ will have period $S$. Thus, for this list of numbers to have size $2n-1$, we must have that $S=p$. Now, assume that $\rho$ has maximal strategic pile. Let $O$ denote the order of $p(K-1)\bmod 2n-1$.

We will show that the strategic pile of $\pi$ has size $p\cdot O$. Since $\rho$ has max strategic pile, the first $p$ elements in the orbit of $x$ in $C_{\pi}$ have one element of each coset $\mod p$; $C_\pi^k(x)=x\mod p$ iff $k=0\mod p$. Let $d_k=\frac{1}{K}(c_k-c_{k+1})$. There exists a permutation $\sigma\in S_p$ such that for each $k\in [p]$, $C_\pi^{k-1}(x)+1-d_{\sigma(k)}=C_\pi^{k}(x)$. Thus, $C_\pi^p(x)=\pi(x)+p-\sum_{k=1}^p d_{\sigma(k)}=\pi(x)+p(K-1)$. Thus, the smallest $kp$ such that $C_{\pi}^{kp}(\pi(n))=\pi(n)$ is the smallest $kp$ such that $kp(K-1)=0\mod n$; by the definition of the strategic pile, $kp$ is also the size of the strategic pile. But $k=O$, as desired.
\end{proof}

Since the case of odd length permutations with maximal strategic pile reduces to the even length case, we are specifically interested in permutation in $M_{2n}$. Using the above characterizations, we can count the members of $M_{2n}$ having periodic difference sequence with period dividing $p|2n - 1$. Since reduced permutations in $M_{2n}$ have length $2n-1$, $p$ must be odd. To count such permutations we must count $\varphi, \mathcal{R}, K$ subject to the conditions in Theorem 3.16 and Theorem 3.17. Counting $\varphi$ is counting permutations of (even) length $p+1$ having maximal strategic pile, which is $2\frac{p!}{p+1}$ by Theorem 3.3 in \cite{GAETZ}. There are $\left(\frac{2n - 1}{p}\right)^p$ choices for $\mathcal R$. Finally, to count the choices of $k$ that yield a permutation that has both periodic difference sequence and maximal strategic pile, we need to define a variant of Euler's totient function. Let 
\[\psi (m) := \{0 < c <= m: gcd(c, m) = gcd(c-1, m) = 1\}\]
It turns out that this function has closed form:
\[\psi(m) = m \prod_{\substack{q \mid m \\ q\text{ prime}}}\left(1 - \frac{2}{q}\right).\]
We have then shown the following:

\begin{corollary}
    Let $n$ be a positive integer and let $p$ be a divisor of $2n-1$. The number of permutations in $M_{2n}$ having periodic difference sequence with period dividing $p$ is
    \[
       2\frac{p!}{p+1} \cdot \left(\frac{2n-1}{p}\right)^p \cdot \psi(n)
    \]
\end{corollary}
\begin{proof}
By Theorem 3.16, $K$ must satisfy $\text{gcd}\left(K,\frac{2n-1}{p}\right) = 1$. By Theorem 3.17, in order for the resulting permutation to have maximal strategic pile, we must also have that $K - 1$ has order $\frac{2n-1}{p}$ in $\mathbb Z_{2n-1}$, which is equivalent to $\text{gcd}\left(K-1,\frac{2n-1}{p}\right) = 1$. Thus $\psi\left(\frac{2n-1}{p}\right)$ is the number of valid choices for $K$. The result then follows from the preceeding exposition. 
\end{proof}

\section{Taxonomy}

The techniques developed in the previous section will now be used to analyze $M_{2n,k}$ for certain values of $n$ and $k$. As an outline of the upcoming work, difference sequences will be used to compute some values $\vert M_{2n,k}\vert \mod (2n-1)^2$. Then, elements of $M_{2n,k}$ are characterized for the following specific values of $k$: $k=\binom{2n-1}{2}$, $k=\binom{2n-1}{2}-4$, and $k=2n-1$. 

\begin{definition}[Compatible Pointers] Let $\pi$ be a permutation, and let $p^*=(p,p+1)$, $q^*=(q,q+1)$ be a pair of pointers. Pointers $p^*$ and $q^*$ are \emph{compatible} if they constitute a valid pointer context in $\pi$.
\end{definition}

\begin{definition}[Universal Pointer] Let $\pi$ be a permutation. A pointer is $\pi$-\emph{universal} if it is compatible with each other pointer. 
\end{definition}

\subsection*{Counting Classes mod $(2n-1)^2$}\leavevmode

In the following theorem and proof, for $d\in\mathbb{Z}_{2n-1}^*$, let $d^{-1}$ denote the multiplicative inverse of $d$ in $\mathbb{Z}_{2n-1}$. Besides this, all arithmetic will be done in $\mathbb{Z}$.

\begin{theorem}Let $\pi \in S_{2n}$ have constant difference sequence with difference value $d$. If $\gcd(2n - 1,d) = 1$, then the number of valid pointer contexts for $\pi$ is $(2n - 1)\cdot \min\left(d^{-1} - 1,2n-1-d^{-1}\right)$.
\end{theorem}

\begin{proof}
Let $p=\pi(1)$. We have $\pi^{-1}(p+1)-1=d^{-1}$ and $\pi^{-1}(p-1)=-d^{-1}. $

First consider the case where $\pi^{-1}(p+1)<\pi^{-1}(p-1)$. Each pointer appearing between $p$ and $p+1$ occurs exactly once, so $p^*=(p,p+1)$ is compatible with $2(d^{-1}-1)$ pointers. If $q\in\mathbb{Z}_{2n-1}$, then we can cyclically shift so that $\pi(1)=q$ without changing the difference sequence. Thus, $q^*=(q,q+1)$ is also compatible with $2(d^{-1}-1)$ pointers. Thus, the total number of available valid pointer contexts is\[\frac{1}{2}((2n-1)\cdot 2(d^{-1}-1)=(2n-1)(d^{-1}-1).\]

Now, consider the case where $\pi^{-1}(p-1)<\pi^{-1}(p+1)$. Each pointer appearing between $p$ and $p-1$ occurs exactly once, so $(p-1,p)$ is compatible with $2(2n-1-d^{-1})$ other pointers. Using an appropriate cyclic shift, we can show that for all $q\in\mathbb{Z}_{2n-1}$, $(q-1,q)$ is compatible with $2(2n-1-d^{-1})$ other pointers. Thus, the total number of available valid pointer contexts is \\ \[\frac{1}{2}(2n-1\cdot 2(2n-1-d^{-1}))=(2n-1)(2n-1-d^{-1}).\]

\end{proof}

\begin{corollary}
Let $\pi$ be a member of $S_{2n}$, where $p = 2n - 1$ is an odd prime number. 
\begin{enumerate}
\item{For $k \in [(2n - 4)/2]$ there are exactly $2p$ permutations with constant difference sequence, maximal strategic pile, and $kp$ valid pointer contexts.}
\item{There are exactly $p$ permutations with constant difference sequence, maximal strategic pile, and $\binom{p}{2}$ valid pointer contexts.}
\item{All other permutations in $S_{2n}$ with max strategic pile have orbit size $p^2$.}
\end{enumerate}
\end{corollary}

\begin{proof}
When $\pi \in S_{2n}$ has nonperiodic difference sequence, its orbit under $\phi'$ is of order $(2n - 1)^2$. The equation $k=\min\left(d^{-1}-1,2n-1-d^{-1}\right)$ has two solutions for $k\in [(2n-1)/4]$ and one solution for $k=\binom{2n-1}{2}$.
\end{proof}

\begin{theorem} Let $\pi\in M_{2n}$ have a difference sequence with period $p$. Then, the number of available valid pointer contexts is a multiple of $\frac{2n-1}{p}$.
\end{theorem}
\begin{proof} Let $K=\pi(p+1)-\pi(1)$. The pointers $(p, p+1)$ and $(q, q+1)$ are compatible if and only if $(p+K, p+K+1)$ is compatible with $(q+K, q+K+1)$. Since the order of $K\mod 2n-1$ is $\frac{2 n-1}{p}$, the result follows.
\end{proof}

\begin{corollary} If $k$ is relatively prime to $2n-1$, then $|M_{2n,k}|=0\bmod (2n-1)^2$.
\end{corollary}
\begin{proof} By the previous theorem, no element of $M_{2n,k}$ can have a periodic difference sequence, so every element of $M_{2n,k}$ has orbit size $(2n-1)^2$.
\end{proof}

\subsection*{Characterizing $M_{2n,k}$ for $k=\binom{2n-1}{2}$}

\begin{proposition} The permutation $[2\ 4\ \dots\ 2n-2\ 2n\  1\ 3\ \dots\ 2n-1]$ has maximal strategic pile and $\binom{2n-1}{2}$ valid pointer contexts. 
\end{proposition}

\begin{lemma} If $\pi\in S_{2n}$ has a universal pointer $(p,p+1)$, then exactly $2n-2$ pointers appear to the right of the leftmost instance of $(p,p+1)$ and to the left of the rightmost instance of $(p,p+1)$.
\end{lemma} 

\begin{proof} Each pointer other than $(p,p+1)$ must appear exactly once the right of the leftmost instance of $(p,p+1)$ and to the left of the rightmost instance of $(p,p+1)$, since each other pointer is compatible for \textbf{cds} with $\pi$. There are $2n-2$ such pointers.
\end{proof}

\begin{theorem} Every $\pi\in M_{2n,\binom{2n-1}{2}}$, is a cyclic shift of  $[2\ 4\ 6\ \dots\ 2n-2\ 1\ 3\ 5\ \dots\ 2n-1]$. \end{theorem} 
\begin{proof} We will show that $\pi$ has constant difference sequence $2$. Let $k\in [2n-1]$, and cyclically shift $\pi$ so that $\pi(1)=k$. Then, since exactly $2n-2$ pointers must appear to the right of the leftmost instance of $(k,k+1)$ and to the left of the rightmost instance of $(k,k+1)$, we have $\pi(n+1)=k+1$. Since exactly $2n-2$ pointers must appear to the right of the leftmost instance of $(k+1,k+2)$ and to the left of the rightmost instance of $(k+1,k+2)$, we have $\pi(2)=k+2$, as desired.
\end{proof}

\begin{corollary} $|M_{2n,\binom{2n-1}{2}}|=2n-1$.
\end{corollary}

\subsection*{Characterizing $M_{2n,k}$ for $k=\binom{2n-1}{2}-4$}

\begin{proposition} For $n\geq 3$, the permutation $[4\ 2\ 6\ \dots\ 2n-2\ 2n\  1\ 3\ 5\ \dots\ 2n-1]$ has maximal strategic pile and $\binom{2n-1}{2}-4$ valid pointer contexts and a non-periodic difference sequence. Thus, there are at least $(2n-1)^2$ permutations in $M_{2n,\binom{2n-1}{2}-4}$.
\end{proposition}

\begin{lemma}\label{evcompat}
For $\pi\in M_{2n}$, each pointer $(p, p+1)$ is compatible with an even number of pointers.
\end{lemma}
\begin{proof} For $p\in [2n-1]$, cyclically shift $\pi$ so that $\pi(1)=p$. Then, since each number occuring after $p$ and before $p+1$ in $\pi$ has two associated pointers, there are an even number of pointers to the right of the leftmost instance of $(p,p+1)$ and to the left of the rightmost instance of $(p,p+1)$. Let $\alpha$ be the multiset of pointers that appear to the right of the leftmost instance of $(p,p+1)$ and to the left of the rightmost instance of $(p,p+1)$. Let $\beta\subset \alpha$ be the set of elements that appear exactly once in $\alpha$. Then, $|\beta|$ is the number of pointers compatible with $(p,p+1)$, and since every other element of $\alpha$ appears twice, $|\beta|$ is even.
\end{proof}

\begin{corollary}\label{evincompat}
For $\pi\in M_{2n}$, each pointer is incompatible with an even number of pointers.
\end{corollary}

\begin{lemma}
If $\pi\in M_{2n}$ has $\binom{2n-1}{2}-c$ valid pointer contexts for $c>0$ and $a$ non-universal pointers, then $\binom{a}{2}\geq c\geq a$.
\end{lemma}
\begin{proof} Let $\pi\in M_{2n}$ have $\binom{2n-1}{2}-c$ valid pointer contexts and $a$ non-universal pointers. Define the ``incompatibility" graph $G_\pi$ as follows: the vertices are the set of non-universal pointers, and draw an edge between two non-universal pointers if they are incompatible. Then, $G_\pi$ has $a$ vertices and $c$ edges. Moreover, each vertex has degree at least $2$. Thus, we have $c\leq \binom{a}{2}$ and $c\geq a$. 
\end{proof}

\begin{lemma}
If $\pi\in M_{2n}$ has $\binom{2n-1}{2}-c$ valid pointer contexts for $c>0$, then $c\geq 3$.
\end{lemma} 
\begin{proof} Let $a$ be the number of non-universal pointers of $\pi$. If $c=1$ or $c=2$, then $a\leq c\leq \binom{a}{2}<a$, which is impossible.
\end{proof}

\begin{lemma}\label{nonuniversalconsecutive} Let $\pi\in M_{2n}$, and let $(p,p+1)$ be a non-universal pointer. Then, either $(p-1,p)$ or $(p+1,p+2)$ is non-universal.
\end{lemma}

\begin{proof} Assume for a contradiction that $(p,p+1)$ is non-universal but $(p-1,p)$ and $(p+1,p+2)$ are both universal. Cyclically shift $\pi$ so that $\pi(1)=p$. Then, we have $\pi(n)=p-1$. Since $(p,p+1)$ is compatible with $(p-1,p)$, we must have $\pi(n+k)=p+1$ for $1\leq k\leq n-1$.  Then, since $(p+1, p+2)$ is universal, we have $\pi(k+1)=p+2$. The sequence of pointers of $\pi$ is as follows: \[[(p-1, p), (p,p+1), \alpha, (p+1,p+2), (p+2, p+3), \beta, (p-2,p-1), (p-1, p), \gamma, (p,p+1), (p+1,p+2), \delta].\] Thus, each pointer appears exactly once in the sequence \[[(p,p+1), \alpha, (p+1,p+2), (p+2, p+3), \beta, (p-2,p-1), (p-1, p)],\] and each pointer appears exactly once in the sequence \[[(p+2, p+3), \beta, (p-2,p-1), (p-1, p), \gamma, (p,p+1), (p+1, p+2)].\] Thus, $\alpha$ and $\gamma$ contain exactly the same pointers. Let $A$ and $C$ be the sequences of elements of $\mathbb{Z}_{2n-1}$ with pointer sequences $\alpha$ and $\gamma$, respectively. If $q\in A$, then $q-1$ and $q+1$ must be in $C$, so $q-2$ and $q+2$ must be in $A$, and so on. But then, depending on the parity of $p-q$, either $p$ or $p-1$ is in $A$, a contradiction.
\end{proof}

\begin{lemma} A permutation $\pi\in M_{2n}$ cannot have a collection of pointers $\{(p,p+1), (p+1, p+2), (p+2,p+3)\}$ which are pairwise incompatible but every other pointer is universal. \end{lemma}
\begin{proof} Assume for a contradiction that $\pi\in M_{2n}$ has this property. Cyclically shift $\pi$ so that $\pi(1)=p$. Since $(p-1,p)$ is universal, we have $\pi(n)=p-1$. Since $(p, p+1)$ is compatible with $2n-4$ pointers, either $\pi(n)=p+1$ or $\pi(n+2)=p+1$; since $\pi(n)=p-1$, we have $\pi(n+2)=p+1$. Since $2n$ pointers appear to the right of the leftmost instance of $(p,p+1)$ and to the left of the rightmost instance of $(p,p+1)$, $(p+1,p+2)$ and $(p+2, p+3)$ must both appear twice in the sequence of pointers to the right of the leftmost instance of $(p,p+1)$ and to the left of the rightmost instance of $(p,p+1)$. We then have  \[\pi=[p\ A\ p-1\ b\ p+1\ C],\] where $A$ is a sequence of $n-2$ elements of $\mathbb{Z}_{2n-1}$, $b\in\mathbb{Z}_{2n-1}$, and $C$ is a sequence of $n-3$ elements of $\mathbb{Z}_{2n-1}$. Let $\alpha$ be the multiset of pointers of elements of $A$, then since $(p-1,p)$ is universal, $\{(p-1,p),(p, p+1), (p-2, p-1)\}\cup \alpha$ contains each pointer exactly once. Then, since $(p,p+1)$ is incompatible with $(p+1,p+2)$ and $(p+2,p+3)$, we have $b=p+2$. But then, $(p-1, p)$ is not compatible with $(p+1, p+2)$, a contradiction.
\end{proof}

\begin{corollary} $\pi\in M_{2n}$ cannot have exactly $\binom{2n-1}{2}-3$ valid pointer contexts. \end{corollary}

\begin{theorem} A permutation $\pi\in M_{2n}$ with $\binom{2n-1}{2}-4$ valid pointer contexts must be in the orbit of $[4\ 2\ 6\ \dots\ 2n-2\ 1\ 3\ \dots\ 2n-1]$.
\end{theorem}

\begin{proof} By lemma \ref{nonuniversalconsecutive}, the non-universal pointers of $\pi$ are of the form \\ $\{(p-1, p), (p,p+1), (q-1, q), (q, q+1)\}$. where without loss of generality $q\neq p+2$. Up to a cyclic shift and a translation, we have $p=4$ and $\pi(1)=4$. Since each non-universal pointer is incompatible with at least two other pointers, and since the number of available valid pointer contexts is $\binom{2n-1}{2}-4$, each non-universal pointer is incompatible with exactly two other pointers. We thus have that either $\pi(n)=5$ or $\pi(n+2)=5$. If $\pi(k)=6$, then since $(5,6)$ is universal, we have $\pi(k+n-1)=5$. Since $(5, 6)$ is compatible with $(4, 5)$, we have $\pi^{-1}(6)<\pi^{-1}(5)$, so $\pi(n+2)=5$ and $\pi(3)=6$. Since $(3,4)$ is compatible with $2n-4$ pointers, we either have $\pi(n-1)=3$ or $\pi(n+1)=3$, so $(3, 4)$ is compatible with $(4,5)$, and both are not compatible with $(q-1, q)$ and $(q,q+1)$. We have that \[\pi=[4\ a\ 6\ B\ 5\ C],\] where $a\in\mathbb{Z}_{2n-1}$, $B$ is a sequence of $n-2$ elements of $\mathbb{Z}_{2n-1}$, and $C$ is another sequence of elements of $\mathbb{Z}_{2n-1}$. Let $\beta$ be the multiset of pointers of $B$. Since $(5,6)$ is universal, very pointer appears exactly once in the multiset $\{(5, 6),(6, 7), (4,5)\} \cup \beta$. Then, $(4,5)$ is incompatible with $(a-1,a)$ and $(a,a+1)$, so $q=a=\pi(2)$. Since $q$ appears between $4$ and $3$, the sequence of pointers to the right of the leftmost instance of $(3,4)$ and to the left of the rightmost instance of $(3,4)$ contains two copies of each of $(q-1,q),(q,q+1)$, and contains all pointers other than $(3,4)$, $(q-1,q)$, and $(q,q+1)$ exactly once. Thus, $\pi(n+1)=3$, and $\pi^{-1}(q+1)\leq \pi^{-1}(3)$. Since $(q,q+1)$ is compatible with $2n-4$ other pointers, we have $\pi(2+n-1)=q+1$ or $\pi(2+n+1)=q+1$. Since $\pi^{-1}(q+1)\leq \pi^{-1}(3)=n+1$, we have $\pi(n+1)=q+1$ and thus $q+1=3$, so $q=2$. For all numbers $c\in \{6,\dots,2n-2\}$, since the pointer $(c,c+1)$ is universal, we have that if $\pi(k)=c$, then $\pi(k+n)=c+1$. Since we know already that $\pi(3)=6$, the rest of the permutation is determined, and is equal to $[4\ 2\ 6\ \dots\ 2n-2\ 1\ 3\ \dots\ 2n-1]$. 
\end{proof}

\begin{corollary} For $n\geq 3$, $|M_{2n,\binom{2n-1}{2}-4}|=(2n-1)^2$.
\end{corollary}

\subsection*{Characterizing $M_{2n,k}$ for $k=2n-1$}
\begin{proposition} For $n\geq 3$, the permutations $[2n\ 1\ n+1\ 2\ n+2\ \dots\ n-1\ 2n-1\ n]$ and $[2n-1\ 2n-2\ \dots\ 2n\  1]$ have maximal strategic pile, $(2n-1)$ valid pointer contexts, and constant difference sequences. Thus, there are at least $2(2n-1)$ permutations in $M_{2n,2n-1}$. \end{proposition}

\begin{lemma}
If $\pi\in M_{2n,2n-1}$, then each pointer is compatible for \textbf{cds} with exactly two other pointers.
\end{lemma} 
\begin{proof} By \cite{ADAMYK}, Theorem 2.19, each pointer must be compatible with some other pointer. By lemma \ref{evcompat}, each pointer must be compatible with an even number of pointers. Since there are $2n-1$ pointers, the result follows.
\end{proof} 

\begin{lemma}\label{stackup} Let $\pi\in M_{2n}$. Assume $(p,p+1)$ is compatible with two other pointers. Let $a=\pi^{-1}(p)<\pi^{-1}(p+1)=b$. Then, $\pi(\{a+1,\dots,b-1\})=\{k,k+1,\dots,k+l\}$ for some $k,l$.
\end{lemma}

\begin{proof} Assume not. Then, there are $q,r\in \pi(\{a+1,\dots,b-1\})$ with $q+1,r+1\notin \pi(\{a+1,\dots,b-1\})$, and there is an $s\in  \pi(\{a+1,\dots,b-1\})$ with $s-1\notin  \pi(\{a+1,\dots,b-1\})$. But then, $(p,p+1)$ is compatible with $(r,r+1)$, $(q,q+1)$, and $(s-1,s)$, a contradiction.
\end{proof}

\begin{lemma}\label{stackdown} Let $\pi\in M_{2n}$. Assume $(p,p+1)$ is compatible with two other pointers. 

Let $a=\pi^{-1}(p+1)<\pi^{-1}(p)=b$. Then, $\pi(\{a,\dots,b\})=\{k,k+1,\dots,k+l\}$ for some $k,l$.
\end{lemma}
\begin{proof} If we cyclically shift $\pi$ to a permutation $\phi$ so that $\pi^{-1}(p)<\pi^{-1}(p-1)$, applying Lemma \ref{stackup} gives that $\mathbb{Z}_{2n-1}\setminus \pi(\{a,\dots,b\})=\{k',k'+1,\dots,k'+l'\}$ for some $k',l'$. Thus, there are $k,l$ such that $\pi(\{a,\dots,b\})=\{k,k+1,\dots,k+l\}$, as desired. 
\end{proof}

Let $\pi$ be a permutation. For $m\geq 1$, call a proper subsequence $[\pi(k),\dots,\pi(k+m)]$ \textit{violating} if $\pi(\{k,k+1,\dots,k+m\})=\{a,a+1,\dots,a+m\}$ for some $a$, if $\pi(k)=a$ and if $\pi(k+m)=a+m$. 

\begin{lemma}\label{violate} If $\pi$ is a permutation with a violating subsequence, then $\pi$ cannot be in $M_{2n}$.
\end{lemma}
\begin{proof} Let $[\pi(k),\dots,\pi(k+m)]$ be the violating subsequence. Assume for a contradiction that $\pi\in M_{2n}$, and let $\pi'$ be the permutation in $S_{2n}$ with maximal strategic pile such that $\pi$ is the reduction of $\pi'$. Then, if $c\in\{\pi(k),\dots, \pi(k+m-1)\}$, we have $C_{\pi'}(c)\in \{\pi(k),\dots, \pi(k+m-1)\}$. Thus, $\pi(k)\notin\SP(\pi')$, so $\pi'$ does not have maximal strategic pile.  
\end{proof}

\begin{theorem} Let $\pi\in M_{2n,2n-1}$ have $\pi(2n-1)=1$. Then, $\pi=[2n-1\ 2n-2\ \dots\ 2\ 1]$ or $\pi=[n+1\ 2\ n+2\ 3\ \dots\ n\ 1]$.
\end{theorem}
\begin{proof} \underline{Case 1}: Assume $\pi^{-1}(3)<\pi^{-1}(2)$. We will show that for all $k$, $\pi^{-1}(k)=2n-k$. 

First, we will show that $\pi(2n-2)=2$. By Lemma \ref{stackdown}, we have \[\pi =[\alpha\  2\ \beta\ 1], \] where for some $a$, $\alpha$ is a permutation of $\{3,\dots,a\}$, and $\beta$ is a permutation of $\{a+1,\dots,2n-1\}$.
Note that we retain the possibility that either $\alpha$ or $\beta$ is empty, and in fact seek to prove that $\beta$ is empty. Assume $\beta$ is nonempty. By applying Lemma \ref{stackup} with the pointer $(a,a+1)$, we have $\pi(\pi^{-1}(2)+1)=a+1$. In other words, $a+1$ is the ``first" element of $\beta$. Then, the subsequence $[\beta\ 1]$ is violating, so $\beta$ must be empty by Lemma \ref{violate}. 

Now, let $m$ be such that for all $l<m$, $\pi(2n-l)=l$. We will show that $\pi(2n-m)=m$. By induction, this will complete the proof of Case 1. By Lemma \ref{stackdown}, we have \[\pi = [\gamma \ m\ \delta \ m-1\ \dots\ 1],\] where $\gamma$ is a permutation of $\{l+1,\dots,2n-1\}$ and $\delta$ is a permutation of $\{m+1,\dots,l\}$. Again, we retain the possibility that either $\gamma$ or $\delta$ is empty, and in fact want to prove that $\delta$ is empty. 

Assume $\delta$ is nonempty. Cyclically shift $\pi$ to a $\phi$ such that $\phi(2n-1)=m$. Then, $a=\phi^{-1}(l)<\phi^{-1}(l+1)=b$, and by Lemma \ref{stackup}, no element of $\{m+1,\dots,l-1\}$ can be in $\phi(\{a+1,\dots,b-1\})$. Thus, $l=\phi(\phi^{-1}(m-1)-1)$, and thus $l=\pi(2n-m)$. In other words, $l$ is the ``last" element of $\delta$. Thus, $[m\ \delta]$ is a violating subsequence, which is impossible by Lemma \ref{violate}. 

\underline{Case 2}: Assume that $\pi$ is not in the orbit of $[2n-1\ 2n-2\ \dots\ 2\ 1]$. We will show that $\pi=[n+1\ 2\ n+2\ 3\ \dots\ n\ 1]$. Let $x\in\mathbb{Z}_{2n-1}$. If we cyclically shift $\pi$ to a permutation $\psi$ such that $\psi(2n-1)=x$, then $\psi^{-1}(x+2)>\psi^{-1}(x+1)$. To see this, if $\psi^{-1}(x+2)<\psi^{-1}(x+1)$, then we could translate $\psi$ by subtracting $x-1$ from every element to get $\psi'$ with the property that $\psi'^{-1}(3)<\psi'^{-1}(2)$ and $\psi'(2n-1)=1$. Then, by the proof of case 1, $\psi'=[2n-1\ 2n-2\ \dots\ 2\ 1]$, but $\pi$ is in the orbit of $\psi'$, so this is impossible. 

Let $x\in\mathbb{Z}_{2n-1}$. We will show that $\pi^{-1}(x+1)-\pi^{-1}(x)=2$. From this, the result will follow. In fact, it suffices to show that $\pi(2)=2$. To see this, cyclically shift $\pi$ to a permutation $\rho$ such that $\rho(2n-1)=x$. Then translate $\rho$ to a permutation $\rho'$ by subtracting every element by $x-1$. Then, $\rho'$ is a permutation in $M_{2n,2n-1}$ with $\rho'(2n-1)=1$ and such that $\rho'$ is not in the orbit of $[2n-1\ 2n-2\ \dots\ 2\ 1]$. Thus, $\rho'(2)=2$, so $\pi^{-1}(x+1)-\pi^{-1}(x)=2$. 

We will show that $\pi(2)=2$. Since $1,2$ do not form an adjacency, We do not have $\pi(1)=2$. Moreover, since $\pi(3)>\pi(2)$, the set of elements $y$ with $\pi^{-1}(2)<\pi^{-1}(y)<\pi^{-1}(1)$ is nonempty.

Let $\pi(k+1)=2$. By applying Lemma \ref{stackup} to an appropriate cyclic shift of $\pi$, we see that $\pi(\{1,\dots,k\})=\{a,a+1,\dots,a+k-1\}$ for some $a$. We will show that $a=a+k-1$, establishing $k=1$, as desired. 

Assume for a contradiction that $\pi^{-1}(a)<\pi^{-1}(a+k-1)$. Since \[\pi^{-1}(a)<\pi^{-1}(a+k-1)<\pi^{-1}(2)<\pi^{-1}(a+k),\] by Lemma \ref{stackup}, we have $\pi^{-1}(a+b)<\pi^{-1}(a+k-1)$ for all $0\leq b<k-1$. Thus, $\pi(k)=a+k-1$. We also have \[\pi^{-1}(a)<\pi^{-1}(a+k-1)<\pi^{-1}(2)<\pi^{-1}(a-1)<\pi^{-1}(1).\] By Lemma \ref{stackdown}, we must have $\pi^{-1}(a)<\pi^{-1}(a+b)$ for every $1\leq b\leq k-1$. Thus, $\pi(1)=a$. Then, $[\pi(1)\ \pi(2)\ \dots\ \pi(k)]$ is a violating subsequence, which is impossible by Lemma \ref{violate}.

Instead, assume for a contradiction that $\pi^{-1}(a+k-1)<\pi^{-1}(a)$. We then have \[\pi^{-1}(a+k-1)<\pi^{-1}(a)<\pi^{-1}(2)<\pi^{-1}(a+k)\leq \pi^{-1}(1).\] By Lemma \ref{stackup}, all elements $c\in \{2,3,\dots,a-1\}$ have $\pi^{-1}(c)<\pi^{-1}(a+k)$. Moreover, any element $d\in\{a+k+1,\dots,2n-1\}$ must have $\pi^{-1}(d)>\pi^{-1}(a+k)$. Then, $[a+k\ \dots\ 1]$ is a violating subsequence, which is impossible by Lemma \ref{violate}.

Thus, $\pi^{-1}(a)=\pi^{-1}(a+k-1)$, so $a=a+k-1$, $k=1$, and $\pi(2)=2$. 

\end{proof}

\begin{corollary} For $n\geq3$, $|M_{2n,2n-1}|=2(2n-1)$.
\end{corollary}

\section{The \textbf{cds} Game}

The following combinatorial game associated with \textbf{cds} was introduced in \cite{ADAMYK}:  

\begin{definition}\label{def:cdsgame}[\textbf{cds} game] For permutation $\pi$ and set $A\subseteq \SP(\pi)$ the two-player game $\CDS(\pi,A)$, called the \textbf{cds} game, is played as follows: 

Player ONE selects a $\pi$-valid pointer context and performs \textbf{cds} on $\pi$ for that pointer context. Let $\pi_1$ be the resulting permutation. Then player TWO selects a $\pi_1$-valid pointer context and performs \textbf{cds} on $\pi_1$ for that pointer context. Let $\pi_2$ be the resulting permutation. ONE and TWO alternate making such moves until a \textbf{cds} fixed point $\phi$ is reached. If $\phi\in A$, ONE wins. Otherwise, TWO wins.
\end{definition}

The number of moves in this game can be pre-computed efficiently from a given permutation $\pi$: For let $c(\pi)$ denote the number of cycles (including length $1$ cycles) in the disjoint cycle decomposition of $C_{\pi}$. The Duration Theorem, Theorem 3.2 of \cite{ADAMYK} (see also Theorem 4 of \cite{Christie}), states
\begin{theorem}\label{thm:cdsduration} For each $\pi\in\textsf{S}_n$ that is not a \textbf{cds} fixed point, the number of consecutive applications of \textbf{cds} that results in a \textbf{cds} fixed point is
\[
\left\{
\begin{array}{ll}
\frac{n+1-c(\pi)}{2} & \mbox{ if $\pi$ is \textbf{cds} sortable} \\
\frac{n+1-c(\pi)}{2} - 1 & \mbox{ otherwise} \\
\end{array}
\right.
\]
\end{theorem}
In the case when $n$ is even and $\pi$ has a maximum sized strategic pile, $c(\pi)=1$, and thus the duration of any instance of the game $\CDS(\pi,A)$ is $\frac{n}{2}-1$.

As the game $\CDS(\pi,A)$ is in the category of finite two-person perfect information combinatorial games, by a classical theorem of Zermelo \cite{ZERMELO}  one of the players has a winning strategy. Determining which player has a winning strategy in a generic instance of the game appears to be of high computational complexity, and no simple criterion is known at this time. A number of sufficient conditions for ONE to have a winning strategy, or for TWO to have a winning strategy, have been obtained in prior work \cite{ADAMYK, JANSEN}. In this section of the paper additional sufficient conditions for player ONE to have a winning strategy in this game are developed, and a connection with classical Sprague-Grundy numbering \cite{SPRAGUE, GRUNDY} in combinatorial games is adapted to this game. 

The following notational conventions will be followed in this section: $\mathcal{M}_{2n}$ denotes the permutations in $S_{2n}$ with maximal strategic pile. Thus, $M_{2n}$ is the set of contracted versions of permutations in $\mathcal{M}_{2n}$.
For positive integer $p$ the symbol $p^*$ denotes the pointer $(p,p+1)$, and the symbol $\sigma_{2n,p}$ denotes the \textbf{cds} fixed point of length $2n$ of the form $\lbrack p+1 \; p+2\; \cdots\; 2n\; 1\; 2\; \cdots p\rbrack$. 

\subsection*{Sprague-Grundy Numbering}

In this subsection the Strategic Pile Retention Theorem, Theorem 2.21 of \cite{ADAMYK} , will be useful:
\begin{theorem}\label{thm:strpileretention}
Let $\pi\in \textsf{S}_n$ be a permutation for which $\SP(\pi)$ has more than one element. Then for any $x\in \SP(\pi)$ there is a $\pi$-compatible pair $(p,q)$ of pointers such that $x\in\SP(\textbf{cds}_{p,q}(\pi))$.
\end{theorem} 

Moreover, as a direct consequence of the Strategic Pile Removal Theorem, Theorem 2.19 of \cite{ADAMYK},
\begin{lemma}\label{lemma:strpileremoval}
For $\pi\in\mathcal{M}_{2n}$ and $\pi$- compatible pointers $p$ and $q$, $\SP(\textbf{cds}_{p^*,q^*}(\pi)) = \SP(\pi)\setminus\{\sigma_{2n,p}(n),\; \sigma_{2n,q}(n)\}$. 
\end{lemma}

Thus, when player ONE executes \textbf{cds} on permutation $\pi\in\mathcal{M}_{2n}$ with a $\pi$-valid pointer context $(p^*,q^*)$, then by Lemma \ref{lemma:strpileremoval} the resulting permutation has strategic pile $\SP(\pi)\setminus\{\sigma_{2n,p}(n),\sigma_{2n,q}(n)\}$.  Now player TWO is confronted with executing  \textbf{cds} on the permutation $\mathbf{cds}_{p^*,q^*}(\pi)$. In this position the subset of the strategic pile corresponding to a win for TWO is $(\SP(\pi)\setminus A)\setminus\{\sigma_{2n,p}(n),\sigma_{2n,q}(n)\})$.

\begin{definition}\label{def:Child}  Let a permutation $\pi\in\textsf{S}_n$, a set $A\subseteq \SP(\pi)$, and a valid pointer context $(p^*,q^*)$ be given.
\begin{enumerate}
\item{The game $\CDS(\mathbf{cds}_{p^*,q^*}(\pi),(\SP(\pi)\setminus A)\setminus\{\sigma_{n,p}(n),\sigma_{n,q}(n)\})$ is said to be a \emph{child of} the game $\CDS(\pi,A)$. }
\item{The \emph{set of children of} $\CDS(\pi,A)$ is the set $\{\CDS(\mathbf{cds}_{p^*,q^*}(\pi),(\SP(\pi)\setminus A)\setminus\{\sigma_{n,p}(n),\sigma_{n,q}(n)\}): (p^*,q^*)\text{ a valid pointer context on }\pi\}$.}
\end{enumerate}
\end{definition}

With the notion of a child of a game defined, an analogue of the notion of the Sprague-Grundy Number \cite{GRUNDY, SPRAGUE} for a combinatorial game can now be defined for games of the form $\CDS(\pi,A)$. 
\begin{definition}\label{def:SpragueGrundy} Let $\pi\in \textsf{S}_n$ and a subset $A$ of $\SP(\pi)$ be given. Define $\SG(\pi,A)$, the Sprague-Grundy number of $\CDS(\pi,A)$ as follows:
\begin{enumerate}
\item{If $\pi$ is a \textbf{cds} fixed point, 
\[
 \SG(\pi,A) = \left\{
    \begin{array}{ll}
      1 & \mbox{ if } \pi\in A\\
      0 & \mbox{otherwise}
    \end{array}
    \right.
\]
}
\item{If $\pi$ is not a \textbf{cds} fixed point, 
\[
 \SG(\pi,A) = \min\{n\ge 0: n\not\in\{ \SG(\sigma,B): \CDS(\sigma,B) \mbox{ a child of }\CDS(\pi,A)\}\}
\]
}
\end{enumerate}
\end{definition}

\begin{lemma}\label{goodlemma} Let $\pi$ be a permutation. Let $p,q$ be $\pi$-compatible pointers such that the set of pointers other than $q$ $\pi$-compatible with $p$ is the same as the set of pointers other than $p$ $\pi$-compatible with $q$. Let $r,s$ be $\pi$-compatible pointers. Then, 
\begin{enumerate}
    \item The set of pointers other than $q$ that are $\mathbf{cds}_{r,s}(\pi)$-compatible with $p$ is the same as the set of pointers other than $p$ that are $\mathbf{cds}_{r,s}(\pi)$-compatible with $q$,
    \item $p$ and $q$ are $\mathbf{cds}_{r,s}(\pi)$-compatible.
\end{enumerate}
\end{lemma}

\begin{proof} By  Theorem 2 of \cite{JANSEN}. In particular, $p$ and $q$ appear in exactly the same columns of the ``master list" $M(r,s)$, so a pointer $t$ is $\mathbf{cds}_{r,s}(\pi)$ compatible with $p$ if and only if it is $\mathbf{cds}_{r,s}(\pi)$ compatible with $q$. This proves (1). Moreover, $p$ and $q$ appear an even number of times (0 or 2) in each column of the ``master list", so they remain compatible in $\mathbf{cds}_{r,s}(\pi)$, proving (2).
\end{proof}

\begin{definition}\label{def:Pexcellent}  Let $P\subseteq P_n$ be a set of pointers.  A permutation $\pi\in\textsf{S}_n$ is $P-$excellent if 
\begin{enumerate}
    \item Any pair of pointers in $P$ is $\pi$-compatible,
    \item Pointers $p \in P_n\setminus P$ are not $\pi$-compatible with any other pointers. 
    \item $\SP(\pi)= P$. 
\end{enumerate}
The $P-$excellent permutations are exactly those which, after reduction of adjacencies, have maximal strategic pile and a maximal number of valid pointer contexts. 
\end{definition} 

\begin{lemma}\label{exclosed} If $\pi\in S_n$ is $P-$excellent, then $\mathbf{cds}_{p,q}(\pi)$ is $P\setminus\{\sigma_{n,p}(n),\sigma_{n,q}(n)\}-$excellent.  
\end{lemma}
\begin{proof} Let $\sigma_{n,r}(n),\sigma_{n,s}(n) \in P\setminus\{\sigma_{n,p}(n),\sigma_{n,q}(n)\}$. The set of pointers other than $r$ $\pi$-compatible with $s$ is the same as the set of pointers other than $s$ $\pi$-compatible with $r$, and $r$ is $\pi$-compatible with $s$ , so by part 2 of Lemma \ref{goodlemma}, $r$ and $s$ are $\mathbf{cds}_{p,q}(\pi)$-compatible. Thus, condition (1) is satisfied.

Applying \textbf{cds} with valid pointer context $p,q$ creates adjacencies at both $p$ and $q$. Moreover, any pointer $r$ that is part of an adjacency in $\pi$ is part of an adjacency in $\mathbf{cds}_{p,q}(\pi)$, so every pointer not in $P\setminus\{\sigma_{n,p}(n),\sigma_{n,q}(n)\}$ is part of an adjacency in $\mathbf{cds}_{p,q}(\pi)$, proving condition 2.

Let $r$ be a pointer not part of an adjacency in $\mathbf{cds}_{p,q}(\pi)$. Then $r$ was not part of an adjacency in $\pi$, so $\sigma_{n,r}(n)\in \SP(\pi)$. By \cite{ADAMYK}, Lemma 2.15, $\SP(\mathbf{cds}_{p,q}(\pi))=\SP(\pi)\setminus\{\sigma_{n,p}(n),\sigma_{n,q}(n)\}$, so $\sigma_{n,r}(n)\in \SP(\mathbf{cds}_{p,q}(\pi))$. Thus, condition (3) is satisfied.
\end{proof}

\begin{remark}
By Lemma \ref{evcompat}, if $\pi$ is $P-$excellent, then $\vert P\vert$ is odd. 
\end{remark}

\begin{theorem}
Let $g_{2m}(a)$ be the Grundy number of the game $\CDS(\pi,A)$, where $\pi\in S_{n}$ is a permutation which is $P-$excellent for a set of pointers $P$ with $|P|=2m-1$ and $|A|=a$.

$g_{2m}(a)=\begin{cases} 0, & a\leq m-2\text{ or } a=m-1\text{ and } m\text{ is odd}\\
1, & a\geq m+1 \text{ or } a=m\text{ and }m\text{ is odd}\\
2, & a=m\text{ or } m-1\text{ and } m\text{ is even}

\end{cases}$

\end{theorem} 
\begin{proof} By Lemma \ref{exclosed}, the children of a $P-$excellent permutation with $|P|=2m-1$ are $P-$excellent permutations with $|P|=2m-3$. We thus proceed by induction.

The base case $m=1$ is clear, since $\pi$ would already be $P-$excellent. Let $m\geq 2$. If $a=2m-1$, then $g_{2m}(a)$ is the minimal excludant of $\{g_{2m-2}(0)\}$. If $a=2m-2$, then $g_{2m}(a)$ is the minimal excludant of $\{g_{2m-2}(0),g_{2m-2}(1)\}$. If $a=0$, then $g_{2m}(a)$ is the minimal excludant of $\{g_{2m-2}(2m-3)\}$. If $a=1$, then $g_{2m}(a)$ is the minimal excludant of $\{g_{2m-2}(2m-3),g_{2m-2}(2m-4)\}$. Otherwise, $g_{2m}(a)$ is the minimal excludant of $\{g_{2m-2}(2m-3-a), g_{2m-2}(2m-2-a),g_{2m-2}(2m-1-a)\}$. The formula then follows by induction. 
\end{proof}

\begin{corollary}\label{xwin} If $\pi$ is $P-$excellent and $|A|>\SP(\pi)-|A|$, then ONE has a winning strategy in $\CDS(\pi,A)$.
\end{corollary}

\subsection*{Winning Strategies for ONE} \leavevmode

There exists a bound on the size of $A$ relative to the strategic pile necessary for a player to have a winning strategy. 

\begin{theorem}[\cite{ADAMYK}, Theorem 4.3]Let $\pi\in S_{n}$ and $A\subseteq$ SP$(\pi)$.
    \begin{enumerate}
        \item If $\vert A\vert \geq \frac{3}{4} \vert$SP$(\pi)\vert$ then ONE has a winning strategy in \textbf{cds}$(\pi,A)$.
        \item If $\vert A\vert < \frac{1}{4}\vert$SP$(\pi)\vert-2$ then TWO has a winning strategy in \textbf{cds}$(\pi,A)$.
    \end{enumerate}
\end{theorem}

Intuitively, increasing the size of $A$ should provide an advantage to ONE. Despite this, examples given in \cite{JANSEN} show that if $\epsilon>0$, there are permutations $\pi$ and subsets $A\subseteq \SP(\pi)$ with $|A|\geq (3/4-\epsilon)|\SP(\pi)|$ for which ONE does not have a winning strategy in $\CDS(\pi,A)$. However, if a certain kind of ``symmetry" is imposed on a large subset of the elements of $\SP(\pi)$, the constant $3/4$ can be improved to be arbitrarily close to $1/2$. We formalize this as follows. 

\begin{definition}[$P-$good] Let $\pi$ be a permutation and $P$ be a set of pointers. Call $\pi$ $P-$good if \begin{enumerate}
    \item Every pair of pointers in $P$ is compatible.
    \item For all pointers $p,q\in P$, if $r$ is another pointer, $r$ and $p$ are compatible if and only if $r$ and $q$ are compatible,
    \item Given any pointer $p$ not part of an adjacency, $\sigma_{n,p}(n)\in\SP(\pi)$.
\end{enumerate} 
\end{definition}

\begin{lemma}\label{popgood} Let $\pi\in \mathcal{M}_{2n}$ and let $P$ be the set of universal pointers in $\pi$. Then, $\pi$ is $P-$good.
\end{lemma}

\begin{proof} Conditions (1) and (2) are satisfied, as each pointer in $P$ is compatible with every other pointer. Condition (3) is satisfied, since the strategic pile is maximal.
\end{proof}

\begin{lemma}\label{goodclosed} Let $\pi\in S_n$ be $P-$good. Then $\mathbf{cds}_{p,q}(\pi)$ is $P\setminus\{\sigma_{n,p}(n),\sigma_{n,q}(n)\}-$good.
\end{lemma}

\begin{proof}  Let $r,s\in P\setminus\{\sigma_{n,p}(n),\sigma_{n,q}(n)\}$. The set of pointers other than $s^*$ compatible with $r^*$ in $\pi$ is the same as the set of pointers other than $r^*$ compatible with $s^*$ in $\pi$. Moreover, $r^*$ and $s^*$ are $\pi$-compatible. Thus, part 2 of Lemma \ref{goodlemma} gives that $r^*$ and $s^*$ are $\mathbf{cds}_{p,q}(\pi)$-compatible, proving condition (1). Moreover, by part 1 of Lemma \ref{goodlemma}, the set of pointers other than $s^*$ $\mathbf{cds}_{p,q}(\pi)$-compatible with $r^*$ is the same as the set of pointers other than $r^*$ $\mathbf{cds}_{p,q}(\pi)$- compatible with $s^*$, so condition (2) is satisfied. 

Let $r$ be a pointer not part of an adjacency in $\mathbf{cds}_{p,q}(\pi)$. Then $r$ was not part of an adjacency in $\pi$, so $\sigma_{n,r}(n)\in \SP(\pi)$. By \cite{ADAMYK}, Lemma 2.15, $\SP(\mathbf{cds}_{p,q}(\pi))=\SP(\pi)\setminus\{\sigma_{n,p}(n),\sigma_{n,q}(n)\}$, so $\sigma_{n,r}(n)\in \SP(\mathbf{cds}_{p,q}(\pi))$. Thus, condition (3) is satisfied.
\end{proof}

\begin{lemma}\label{bigmovesbigpop} Let $b>0$ and let $n$ be such that $2n-1>b$. If $\pi\in \mathcal{M}_{2n}$ has at least $\binom{2n-1}{2}-b$ valid pointer contexts, then there are at least $2n-1-b$ universal pointers.
\end{lemma}
\begin{proof} By Corollary \ref{evincompat}, each non-universal pointer is incompatible with at least two other pointers. Since there are at most $b/2$ pairs of incompatible pointers, there are at most $b$ non-universal pointers, and therefore there are at least $2n-1-b$ universal pointers.
\end{proof}

\begin{theorem}\label{wincon} Let $\pi\in S_n$ be $P-$good, where $\vert \SP(\pi)\vert=k$. Write $c=k-\vert P\vert$. Assume further that $\vert P\vert>3c$, and let $A\subset \SP(\pi)$ satisfy $\vert A\vert\geq \left\lfloor\frac{k}{2}\right\rfloor+1+2c$. Then, ONE has a winning strategy in $\CDS(\pi,A)$.
\end{theorem}
\begin{proof} If $c=0$, then every pointer is either in $P$ or part of an adjacency, so $\pi$ is $P-$excellent. We have $2\vert A\vert>k$, so $\vert A\vert>k-\vert A\vert$, and ONE has a winning strategy by Corollary \ref{xwin}.

Now, assume $c\geq 1$. Let $p^*$ be a pointer that is not part of an adjacency and is also not in $P$. 
Since $\sigma_{n,p}(n)\in \SP(\pi)$, the Strategic Pile Removal Theorem (\cite{ADAMYK}, Theorem 2.19), there is a pointer $q^*$ such that $p^*$ and $q^*$ are compatible. 
Thus, let ONE perform \textbf{cds} with valid pointer context $p^*,q^*$. 
Let TWO perform \textbf{cds} with valid pointer context $r^*,s^*$, and let $\pi'=\mathbf{cds}_{r^*,s^*}(\mathbf{cds}_{p^*,q^*}(\pi))$. 
Let $A'=A\cap \SP(\pi')$, let $k'=\vert \SP(\pi')\vert$, let $P'=P\setminus\{\sigma_{n,p}(n), \sigma_{n,q}(n),\sigma_{n,r}(n), \sigma_{n,s}(n)\}$, and let $c'=k'-|P'|$. 
It suffices to show that ONE has a winning strategy in $\CDS(\pi',A')$. 
By Lemma \ref{goodclosed}, $\pi'$ is $P'-$good. 
Moreover, since $p\notin P$, $|P|-|P'|\leq 3$. 
We have $\{\sigma_{n,p}(n),\sigma_{n,q}(n),\sigma_{n,r}(n),\sigma_{n,s}(n)\}\subset\SP(\pi)$, and $\SP(\pi')=\SP(\pi)\setminus\{\sigma_{n,p}(n),\sigma_{n,q}(n),\sigma_{n,r}(n),\sigma_{n,s}(n)\}$ by \cite{ADAMYK}, Lemma 2.15. 
Thus $k-k'=4$. Then, \[c'=k'-|P'|\leq k-4-|P|+3= c-1.\] 
Moreover, since $A\setminus A'=A\cap (\SP(\pi)\setminus \SP(\pi'))$, we have $|A|-|A'|\leq 4$. 
Thus, \[|P'|\geq |P|-3>3c-3\geq 3c',\] and \[|A'|\geq |A|-4\geq \left\lfloor \frac{k}{2} \right\rfloor+1+2c-4\geq \left\lfloor \frac{k'+4}{2} \right\rfloor +1+2(c'+1)-4=\left\lfloor \frac{k'}{2} \right\rfloor+1+2c'.\] 
By induction, ONE has a winning strategy in $\CDS(\pi',A')$.
\end{proof}

\begin{corollary} Let $b>0$. Choose $n$ such that $2n-1-b>3b$, and let $\pi\in \mathcal{M}_{2n}$ have at least $\binom{2n-1}{2}-b$ valid pointer contexts. If $A\subset \SP(\pi)$ satisfies $|A|\geq n+2b$, then ONE has a winning strategy in $\CDS(\pi,A)$.
\end{corollary}
\begin{proof} Let $P$ be the set of universal pointers of $\pi$. By Lemma \ref{popgood}, $\pi$ is $P-$good. By Lemma \ref{bigmovesbigpop}, $|P|\geq 2n-1-b$. Let $c=b$, and let $k=2n-1=|\SP(\pi)|$. We have $|P|\geq 2n-1-b>3b=3c$, and $|A|\geq n+2b=\left\lfloor \frac{2n-1}{2}\right\rfloor+1+2c$. Thus, by Theorem \ref{wincon}, ONE has a winning strategy.
\end{proof}

\begin{corollary} Let $b>0$. For all $r>1/2$, there exists an $N$ such that for all $n\geq N$, if $\pi\in \mathcal{M}_{2n}$ has at least $\binom{2n-1}{2}-b$ available valid pointer contexts and $A\subset \SP(\pi)$ has $|A|\geq r|\SP(\pi)|$, then ONE has a winning strategy in $\CDS(\pi,A).$
\end{corollary}

\section{Future Work}
\begin{enumerate}
    \item Compute the number of valid pointer contexts from a periodic difference sequence. Given this result, it would be possible to compute the number of permutations in $S_{2n}$ with maximal strategic pile and a given number of valid pointer contexts $\bmod (2n-1)^2$.
    \item Classify permutations with other strategic pile sizes.
    \item Prove the boundedness or exhibit the unboundedness of Grundy numbers of \textbf{cds} games on permutations with maximal strategic pile.
\end{enumerate}

\section{Acknowledgments}
This research, conducted at the REU CAD site at Boise State University, was funded by NSF Grant DMS-1659872 and Boise State University.

\end{document}